\documentclass{article}
\usepackage[utf8]{inputenc}
\usepackage[main=english]{babel}
\usepackage{mathtools,amssymb,amsthm,upgreek,dsfont,mathabx}
\usepackage{geometry}
\usepackage{enumitem}
\usepackage{amsmath}
\usepackage{url}
\usepackage{fancybox}
\usepackage{fullpage}
\usepackage{natbib}
\usepackage[colorlinks, linktocpage, citecolor=blue, linkcolor=blue]{hyperref}
\usepackage{multicol} 
\usepackage{array}
\usepackage{xargs}
\usepackage{float}
\usepackage{diagbox}
\usepackage{xcolor}
\usepackage{graphicx}
\usepackage[prependcaption]{todonotes}
\usepackage[capitalise]{cleveref}
\usepackage{soul}
\usepackage{tikz}
\usetikzlibrary{automata, positioning, arrows.meta}

\newcommandx{\Max}[2][1=]{\todo[inline, author={Maxime}, linecolor=green,backgroundcolor=green!25,bordercolor=green,#1]{#2}}

\newtheorem{theorem}{Theorem}[section]

\newtheorem{prop}{Proposition}[section]
\newtheorem{lem}{Lemma}[section]

\newtheorem{rem}{Remark}[section]
\renewcommand{\P}{\mathbb{P}}

\newcommand{\E}{\mathbb{E}}
\newcommand{\ind}{\mathds{1}}

\title{Multi-state model with temporal-consistent survival analysis for homogeneous Markov chains}

\author{
{\large Mikael Escobar-Bach}\footnote{Email: \texttt{mikael.escobar-bach@univ-angers.fr}.  \vspace*{.1cm}} \\ \textit{LAREMA, Université d'Angers} \\[.2cm]
{\large Alexandre Popier}\footnote{Email: \texttt{alexandre.popier@univ-lemans.fr; }.  \vspace*{.1cm}} \\ \textit{Laboratoire Manceau de Mathématiques, Le Mans Université}\\[.2cm]
{\large Malo Sahin}\footnote{Email: \texttt{malo.sahin@univ-angers.fr}.  \vspace*{.1cm}} \\ \textit{LAREMA, Université d'Angers}
}

\begin{document}

\maketitle

\begin{abstract}
In this study, we consider sequences drawn from time-homogeneous Markov chains and introduce a novel approach for estimating first hitting-time distributions to specified terminal states. Our methodology is based on the temporal-consistent survival analysis that facilitates the construction of consistent estimators of the distributions from any estimates of the transition rate and transition probabilities. In this line of work, we also discuss the issue of cured individuals with chains that never reach a terminal state, and propose an estimator of the cure rate. Furthermore, we derive non-asymptotic theoretical guarantees for our approach and apply our methodology with kernel type estimators. The latter approach is illustrated in a simulation study using generic data and a real-life application involving patients undergoing bone marrow transplants.
\end{abstract}

\textit{Some key words :} Survival analysis, multi-state model, first hitting-time model, kernel estimation.

\section{Introduction}

Survival analysis focuses on modeling and study time-to-event data and have seen substantial developments in the recent literature due to a wide range of applications. Among others, examples arise in medicine with clinical studies of patients' survival times, in liability with component's functional lifespans prior to failure or in sociology with elapsed unemployment times preceding the finding of a new employment. However, the times to the event of interest may not be fully observed, meaning the corresponding data is subject to censoring. This may occur, for example, when a clinical study has a predetermined follow-up period or when a patient withdraws from the study before the event occurs. \\
A variety of experimental designs enable longitudinal data to be observed in state systems. However, standard survival analysis models are inherently limited as they do not account for sequences of intermediate events that may precede the event of interest. In simple cases involving a single transition step, competing risk models \cite{And02,Put02} consider experiments in which individuals are subject to more than two different events. In this context, multi-state models extend the analysis to stochastic processes that cause individuals to transition between multiple states over time. This particularly offers a flexible and powerful modeling framework that accounts for the time through transient states that an individual or a system may evolve before reaching the state of interest. 
The relevance of such models in survival analysis is readily illustrated by concrete examples. For instance, a patient who has recovered from a disease can not subsequently die from it without first experiencing a relapse or an intermediate state, and ignoring these intermediate could compromise the coherence of the model description. In the literature, the framework of multi-state models is well described, as introduced in \cite{Hou00,And02b,Hou99}. Overall, many studies focus on estimating transition intensities, representing the instantaneous hazard of progressing to a designated state, given that one is occupying another state. When involved in a censoring framework, different assumptions can be made on the dependence between transition rates and times (see the reviewing work in \cite{Mei09}). In particular, time-homogeneous models where transition intensities are constant over time that have been discussed in the early works \cite{Kay86,Chi68} and are still considered in recent work such as \cite{Mac21}. For time-homogeneous Markov processes, a closed-form likelihood can be derived from Kolmogorov forward equations, which admits explicit matrix-exponential solutions for the transition intensity matrix. Other model extensions only assume the Markov property of the processes. However, as no explicit solution of the transition matrix exists in the non-homogeneous Markov case, several alternative methods have emerged. A first approach consists on considering piecewise constant transition intensities and to apply the time homogeneous methodology on each section \cite{Per01}. Depending on the censoring mechanism, alternative non-parametric approaches have also been proposed \cite{Aal78, Fry95}. These approaches have shown great interest as they overcome the limitations imposed by the piecewise constant assumption on transition intensities \cite[e.g.][]{Jol02, Tit11, Ken25}. Finally, few attempts to build efficient estimators in the non-Markovian context have been proposed \cite[see][]{Pep91,Str98,Mei06} which remains an active subject of interest, as seen in the recent works \cite{Una15,Mal21,And22}.\\ 
In parallel to multi-state models, dynamic prediction generalizes classical survival analysis by incorporating longitudinal information to update risk estimates over time. In this setting, predictions of the future events or states are continuously updated over time, conditional on the individual's current state and observed history. Rather than relying solely on baseline covariates, this framework incorporates the evolving history of an individual, such as intermediate events or longitudinal measurements, to refine the estimation of time-to-event outcomes. The review articles \cite{Sur17,Li23} provide a comparison between two mainstream approaches, namely joint modeling and landmarking. Joint models have undergone a substantial development in the recent literature \cite[e.g.][]{Tsi97,Tsi04,Riz11,Riz17,Asa15}. They simultaneously specify a model for the longitudinal process (e.g. repeated biomarker measurements) and a model for the time-to-event outcome, linking the two through an association structure. In contrast to joint models, landmarking relies on a sequence of prediction models defined at a set of pre-specified landmark times and does not require specifying the full joint distribution of the processes \cite{Hou07,Put17,Put22}. However, both approaches present limitations with modeling assumptions, possible loss of information on the full longitudinal history, and substantial computational burden in complex settings. To address these issues, deep learning-based methods such as proposed in \cite{Lin22} and \cite{May22} have recently emerged as flexible alternatives, aiming to relax the required assumptions and improve predictive accuracy while reducing computational costs.\\

\noindent 
In this article, we propose a novel methodology for first hitting-time experiments based on multi-state models with time-homogeneous Markov processes. The underlying Markov chain is endowed with a time-independent random covariate that influences both the transition intensities and the state-transition probabilities. In the recent work \cite{May22}, the authors introduced a deep learning methodology to study first hitting-time distributions with discrete-time Markov chains, relying on the property of temporal consistency. The model we develop extends this framework. In addition, it provides a natural extension to the continuous-time setting and relaxes several underlying assumptions. In particular, our approach allows for a terminal subset of the state space, rather than a single terminal state, and does not require prior specification of this subset for estimation. The approach we propose for estimating the distribution of the first hitting-time also relies on the property of temporal consistency from which an estimator of the density is derived. It follows a kernel based approach from which we study the non-asymptotic properties and provide sharp rates of convergence. In particular, the method is shown to be almost optimal in line with the minimax optimality framework. All in all, our approach allows to relax the proportional hazard assumption which is commonly used when dealing with time-homogeneous Markov multi-state models \cite[see][]{Mac21,Lou17} which is not always verified in practice. In real-world applications, there exist scenarios in which the visitation of certain states may obstruct the process from reaching the terminal state. To account for such configurations, our proposed model also enables absorbing subsets from the state space. Absorbing states may notably correspond to recovery states, making the estimation of the absorbing probabilities particularly relevant. In survival analysis, this phenomenon is referred to as cured individuals, and has attracted a large body of literature, as illustrated in the review article \cite{Ami18}. In line with this framework, we also propose a non-parametric estimator of the probability that the terminal subset is not reached, which equivalently returns an estimation of the cure rate.\\

\noindent
The remainder of the paper is given as follows. In Section \ref{sec::Modeldescrip}, the stochastic construction of the model is presented and an estimator of the first hitting-time distribution is proposed. Based on consistent estimators of the chain transition rate and the transition matrix, bounds for the prediction error of the density estimator and the absolute error-risk  of the cure rate are derived. Section \ref{Sec::Kernel} proposes kernel based estimators of the chain transition rate and the transition matrix that induce an almost minimax optimal convergence rate for the risk error of the first hitting-time density estimator. The finite sample performances of the estimation procedure are illustrated on simulated datasets in Section \ref{Sec::Simulation} and we highlight its practical applicability with a study on bone marrow transplant data. All proofs and technical lemmas are postponed to Section \ref{Sec::Proof}.

\section{Model description and inference} \label{sec::Modeldescrip}

\subsection{Model}

Let $X=(X_t,\,t\in\mathbb{R}_+)$ be a continuous Markov chain embedded with a time-independent random covariate $Z\in\mathcal{Z}$ in $\mathbb{R}^p$  and with density function $f_Z$. We assume the state space $\mathcal{X}$ to be finite. We also define $\lambda_z>0$ the conditional chain transition rate of the jumping times $(E_i)_{i\in\mathbb{N}}$ where $E_0=0$ almost surely. Note that unlike any generic Markov chain, we impose that the transition rate is independent from the chain positions. In other words, for any $i\geq 1$, $X_{E_i}$ might jump from its current state after a holding time $E_i-E_{i-1}$ which follows an exponential distribution with parameter $\lambda_z$ and independently from the past of the chain. For any fixed and possibly unknown terminal subset $A\subset\mathcal{X}$, the hitting time of interest is given by
\begin{eqnarray*}
T=\inf\{t\geq0,\,X_t\in A\}
\end{eqnarray*}
which represents the first time that the Markov chain enters the terminal subset $A$. Considering the jump chain $Y=(Y_n=X_{E_n})_{n\in\mathbb{N}}$ associated to $X$, we can define the number of jumps before $Y$ reaches the terminal state by 
\begin{eqnarray*}
    S=\inf\{n\geq0,\,Y_n\in A\}
\end{eqnarray*}
where its discrete probability distribution is characterized by the coefficients
\begin{eqnarray*}
    c_j(x,z)=\mathbb{P}(S=j|X_0=x,Z=z),\quad j\in\mathbb{N}.
\end{eqnarray*} 
Due to practical limitations, only a limited number of coefficients can be estimated. However, it is worth mentioning that, for a finite set $\mathcal{X}$, the relatively rapid decrease of the probabilities as $j$ increases enables us to consider only a few coefficients. This is especially useful for controlling the theoretical precision of the density function approximation. In this perspective, the next lemma provides an upper bound for the rate of convergence of the coefficients.    

\begin{lem} \label{lem: expo_decrease}
Assume that $\mathcal{X}$ is a finite state space. For any $z \in \mathcal{Z}$, there exist $m_z>0$, $0<r_z<1$ and $k_0\in\mathbb{N}^*$ such that for any $j\geq k_0$, $\displaystyle \sup_{x\in\mathcal{X}}c_j(x,z)\leq m_z r_z^{j}$.
\end{lem}

In general, it is however uncertain that any chain trajectory surely attains the subset, hence resulting in censored observation. In particular, recordings from real data experiments typically have a limited number of observations per trajectory. We address this limitation by assuming that each chain stops at a random number of observations $M=L\wedge S$, where $L$ is a limiting step in $\mathbb{N}$, independent of $S$ given $Z$. This setting is similar to the independent censoring assumption in survival analysis, which assumes independence between the censoring mechanism and time-to-event outcomes. Formally, each piece of data from one individual is grouped as  $\left(E_M,\delta,(Y_i)_{i=1,\ldots,M},Z\right)$ where $\delta=\mathds{1}_{\{Y_M\in A\}}=\mathds{1}_{\{ S\leq L\}}$ is a censoring indicator telling wether the chain reaches the terminal subset or not. In the sequel, we thus consider samples of size $n$ given by
\begin{eqnarray*}
    \left\{\left(E^{(j)}_{M_j},\delta_j,(Y^{(j)}_i)_{i=0,\ldots,M_j},Z_j\right),\,j=1,\ldots,n\right\}.
\end{eqnarray*}

\noindent
\textbf{Notation} : Throughout the paper and for any covariate position $z\in\mathcal{Z}$, we fix the notation $\mathbb{P}_z=\mathbb{P}(.|Z=z)$ and  $\mathbb{E}_z=\mathbb{E}(.|Z=z)$ to ease the reading.

\subsection{Temporal consistency}
Given the covariate $Z=z$ and an initial state $X_0=x$, our main objective is to infer the hitting-time distribution from its conditional density function $f_T:\mathbb{R}_+\times\mathcal{X}\times\mathcal{Z}\to\mathbb{R}_+$. The stochastic construction of the model allows to decompose the distribution of $T$ throughout the number of jumps of $Y$ and the lengths of the holding times. Specifically and for any $x\in\mathcal{X}$, the density function of $T$ associated to the sub-distribution function $t\mapsto\mathbb{P}_z(T\leq t\;|\;Y_0=x)$ has the form
\begin{eqnarray*}
    f_T(t,x,z)=
   \sum_{j\geq1}c_j(x,z)\dfrac{\lambda_z^jt^{j-1}e^{-\lambda_zt}}{(j-1)!},\quad t\in\mathbb{R}_+\,z\in\mathcal{Z}.
\end{eqnarray*}
In order to estimate the hitting-time probabilities of the jump chain, we introduce an algorithm based on the temporal-differential learning framework as proposed in \cite{May22}. In our procedure, we propose an estimator of the density function by replacing the unknown model coefficients with their empirical counterparts, with
\begin{eqnarray*}
    f_{T,n}^{(k)}(t,x,z)=
    \sum_{j=1}^kc_{n,j}(x,z)\dfrac{\lambda_{n,z}^jt^{j-1}e^{-\lambda_{n,z}t}}{(j-1)!},\quad t\in\mathbb{R}_+\,z\in\mathcal{Z}.
\end{eqnarray*}
where $k\in\mathbb{N}^*$ represents the number of estimated coefficients. In general, any consistent estimator of the hitting-time coefficients could result in an efficient estimation. In this work, we also propose a method based on the Markov property of the jump chain, i.e.
\begin{eqnarray*}
    c_j(x,z)
    =\left\{
    \begin{matrix}
    \displaystyle\sum_{x'\in\mathcal{X}}p_z(x,x')c_{j-1}(x',z)\bold{1}_{\{x\notin A\}}& j\geq 1,\\
    \bold{1}_{\{x\in A\}}&\text{if } j=0.
    \end{matrix}
    \right.
\end{eqnarray*}
Let $p_{n,z}$ denote a proper conditional estimator of the probabilities $p_z$, in the sense that for any $n\in\mathbb{N}^*$ and $z\in\mathcal{Z}$, $(p_{n,z}(x,x'))_{x,x'\in\mathcal{X}}$ is a transition matrix. Since the subset $A$ is possibly unknown, we define $$A_n=\{Y^{(j)}_{M_j};\,\delta_j=1,\,j=1,\ldots,n\},$$ the observed elements from the terminal subset. The computation of the $j$-th coefficient relies on the previous induction and only the first $k$ terms are considered during the procedure, so that our estimators are given by
\begin{eqnarray}
\label{temporalsa}
    c_{n,j}^{(k)}(x,z)
    =\left\{
    \begin{matrix}
    0&\text{if } j>k,\\[.2cm]
    \displaystyle\sum_{x'\in\mathcal{X}}p_{n,z}(x,x')c_{n,{j-1}}^{(k-1)}(x',z)\bold{1}_{\{x\notin A_n\}}& 1\leq j\leq k,\\
    \bold{1}_{\{x\in A_n\}}&\text{if } j=0.
    \end{matrix}
    \right.
\end{eqnarray}
In the following proposition, we propose and prove upper bounds for the risk error of our procedure for any generic estimators of the coefficients or based on the temporal survival analysis.

\begin{prop}
\label{prop::density}
Let $\lambda_{n,z}$ and $(c_{n,j}(x,z))_{1\leq j\leq k}$ be any consistent estimators of $\lambda_z$ and $(c_j(x,z))_{1\leq j\leq k}$ respectively. Also assume that $(c_{n,j}(x,z))_{1\leq j\leq k}$ is a non-negative sequence with $\displaystyle \sum_{j=0}^{k} c_{n,j}(x,z) \leq 1$. Then for $k$ large enough
\begin{eqnarray*}
    &&\dfrac{1}{3} \mathbb{E}\left[\int_0^{+\infty}\sup_{x\in\mathcal{X}}\left( f_{T,n}^{(k)}(t,x,z)- f_T(t,x,z)\right)^2dt\right]\\
    &&\leq\dfrac{(k+1)^{3/2}}{2} \mathbb{E}\left[\dfrac{(\lambda_z-\lambda_{n,z})^2}{\lambda_z+\lambda_{n,z}}\right] + \dfrac{2k-1}{2}\lambda_z\mathbb{E}\left[\sup_{x\in\mathcal{X},\,j=1,\ldots,k}\left(c_{n,j}(x,z)-c_j(x,z)\right)^2\right] +\dfrac{\lambda_zm_z^2r_z^{2k+1}}{2(1-r_z)}.
\end{eqnarray*}
Furthermore, if $c_{n,j}(x,z)=c^{(k)}_{n,j}(x,z)$ as described in (\ref{temporalsa}), then
\begin{eqnarray}
\label{theo::temporalinequality}
    &&\nonumber\dfrac{1}{2}\mathbb{E}\left[\sup_{x\in\mathcal{X},\,j=1,\ldots,k}\left(c^{(k)}_{n,j}(x,z)-c_j(x,z)\right)^2\right]\\
    &\leq& \left(\sum_{x\in\mathcal{X}}\mathbb{P}_z(S\leq k-1|Y_0=x)\right)^2\mathbb{E}\left[\sup_{x,x'\in\mathcal{X}}\left(p_{n,z}(x,x')-p_z(x,x')\right)^2\right]\\
    \nonumber&+&(k+1)^2\left[\sum_{x\in A}\mathbb{P}(Y_S\neq x,S\leq L)^n+\mathbb{P}(S>L)^n\right].
\end{eqnarray}
\end{prop}

\begin{rem}
    Intuitively, the assumptions on $(c_{n,j}(x,z))_{1 \leq j \leq k}$ ensure that the sequences return proper probabilities on mutually disjoint events. Furthermore, despite the apparent arbitrariness of the choice of $k$, the discrepancy between $f_T$ and $f_{T,n}^{(k)}$ of order $r_z^{2k}$ offers practical insights about the control of the error of approximation. For instance, a comparison between our machine error and the latter rate of convergence is used in the simulation section to select the number of iterations.  
\end{rem}

\subsection{Absorbing set and cure rate} \label{subsec:: absorbcure}
In practice, it is common for a segment of the population at risk to never experience the event of interest. In our case, this represents the fraction of chains that never reach the terminal set A and that may remain trapped in an isolated subset. Consequently, the time $S$ may become infinite and be censored. In the context of survival analysis, this phenomenon is referred to as "cured individuals" and often hinders the distinction between them and individuals with insufficient follow-up. In order to consider such scenario in our settings,  we introduce for any covariate position $z\in\mathcal{Z}$
$$B_z := \left\{ x \in \mathcal{X} \backslash A,\, \exists (y,i) \in A \times \mathbb{N}^*\text{ with } p_z^i (x,y) > 0  \right\}$$ the subspace of vertices connected to the terminal set $A$, and $\displaystyle C_z := \mathcal{X} \backslash (A \cup B_z)$, the subspace of isolated vertices. The next Lemma shows that, provided the kernel transition estimation excludes non-connected points, only information about the estimators' efficiency on $B_z$ is required.

\begin{lem}
\label{lem::isolated}
Let $p_{n,z}$ be any conditional estimator of the kernel transition function such that for $n$ large enough and any $x,x'\in\mathcal{X}$
\begin{eqnarray*}
    p_z(x,x')=0 \Longrightarrow p_{n,z}(x,x')=0,\text{ a.s.}
\end{eqnarray*}
Then for any $x\in C_z$, $c_{n,j}^{(k)}(x,z)=0$ and in (\ref{theo::temporalinequality}), the set $\mathcal{X}$ can be replaced by $B_z$. 
\end{lem}
\noindent
There is also particular interest in estimating the proportion of cured individuals, i.e., the probability that $T=+\infty$. Based on the estimation of the density of $T$, we have that 
\begin{eqnarray*}
    \mathbb{P}_z(T=+\infty|Y_0=x)=\left\{
    \begin{matrix}
        0    &x\in A,\\
        1-\sum_{j\geq1}c_j(x,z)& x\in B_z,\\
        1&x\in C_z.
    \end{matrix}
    \right.
\end{eqnarray*}
We thus propose an estimator of the cure rate replacing the unknown coefficients with their empirical counterparts and show that the risk is controlled by the risk of the density estimator. For any $x\in\mathcal{X}$ and $z\in\mathcal{Z}$, we define $R_{n}^{(k)}(x,z)$ the cure rate estimator given by
\begin{eqnarray*}
    R_{n}^{(k)}(x,z)=\left\{
    \begin{matrix}
            1-\sum_{j\geq1}c_{n,j}^{(k)}(x,z)& \text{ if } x\notin A_n.\\
            0& \text{ if }x\in A_n
    \end{matrix}
    \right.
\end{eqnarray*}

\begin{lem}
\label{lem::curerate}
Under the assumptions of Proposition \ref{prop::density}, we have for $k$ large enough
\begin{eqnarray*}
    \mathbb{E}\left[\sup_{x\in\mathcal{X}}\left(R_{n}^{(k)}(x,z)-\mathbb{P}_z(T=+\infty|Y_0=x)\right)^2\right]&\leq& 2 k^2 \mathbb{E}\left[\sup_{x\in\mathcal{X},\,j=1,\ldots,k}\left(c^{(k)}_{n,j}(x,z)-c_j(x,z)\right)^2\right]+2m_z^2\dfrac{r_z^{2(k+1)}}{(1-r_z)^2}\\
    &&+2\sum_{x\in A}\mathbb{P}(Y_S\neq x,S\leq L)^n+\mathbb{P}(S>L)^n.
\end{eqnarray*}
\end{lem}
\noindent
Note that in the previous lemma, the rate of convergence of order $k^2$ against the risk coefficient estimator is clearly slower than the expected rate of $k^{3/2}$ obtained in the density function estimation. Indeed, there is no direct relationship between the aforementioned two risks.

\section{Kernel based approach} \label{Sec::Kernel}

In this section, we propose a weighted moment methods where the transition rate and Markov coefficients are locally approximated via empirical kernel-type statistics. This approach offers a fully data-driven procedure that is known to work well with low dimensions $p$. Furthermore, our procedure will show that optimal convergence rates can be achieved under mild conditions. For this method, we impose $\lambda_z$ to be uniformly bounded, so that there exist$\lambda_{\min},\lambda_{\max}>0$ with $\lambda_{\min}\leq\lambda_z\leq\lambda_{\max}$ for any $z\in\mathcal{Z}$. We also impose some Hölder-type regularity conditions on the model functions due to the regression analysis.

\paragraph{Assumption $(\mathcal{H})$}: There exist $ 0 < \alpha\le  1$ and a constant $C>0$ such that for any $z,z' \in \mathcal{Z}$
\begin{itemize}
    \item[-] $ |f_Z(z)-f_Z(z')|\leq C\Vert z-z'\Vert^\alpha,$
    \item[-] $ |\lambda_z-\lambda_{z'}|\leq C\Vert z-z'\Vert^\alpha,$
    \item[-] $ |p_z(x,x')-p_{z'}(x,x')|\leq C\Vert z-z'\Vert^\alpha$, for any $x,x'\in\mathcal{X}$.\\
\end{itemize}

\noindent
Let $K$ be any positive kernel function taking values in the unit ball of $\mathbb{R}^p$ with respect to any norm $\Vert .\Vert$ and of integral equal to $1$. For any $z\in\text{int}(\mathcal{Z})$ with $f_Z(z)>0$, we introduce the weights 
\begin{eqnarray*}
    \omega_{h,z}(z')=\dfrac{K_h(z-z')}{\sum_{i=1}^nK_h(z-Z_i)},\quad\forall z,z'\in\mathcal{Z}
\end{eqnarray*}
where $K_h(.)=K(./h)/h^p$ and $h=h_n$ is a non-random sequence such that $h\to 0$ and $nh^p\to +\infty$ as $n\to+\infty$. We also impose that the kernel function is bounded away from zero over its support, i.e. one can find positive constants $c_K$ and $C_K$ such that $c_K\leq K(u)\leq C_K$ for any $u$ in the support of $K$. The conditional transition rate estimator is then given by
\begin{eqnarray*}
    \lambda_{n,z}=\left(\sum_{j=1}^n\omega_{h,z}(Z_j)\dfrac{E_{M_j}}{M_j}\right)^{-1}\wedge\tilde{\lambda}
\end{eqnarray*}
where $\tilde{\lambda}$ is any constant large enough so that $\tilde{\lambda}\geq\lambda_{\max}$. 
\begin{rem}
    In practice, limiting the rate estimator below an arbitrary and large enough threshold does not impose any constraints, but it is useful for controlling the quadratic error. This prevents the estimator from considering experiments where the hitting times and/or the number of covariates near $z$ are relatively small.
\end{rem}
\noindent
 In this perspective, we first introduce an estimator of the transition probabilities of the jump chain given by
\begin{eqnarray*}
    p_{n,z}(x,x')=\dfrac{\sum_{j=1}^n\omega_{h,z}(Z_j)N_j^{x,x'}}{\sum_{j=1}^n\omega_{h,z}(Z_j)N_j^x}
\end{eqnarray*}
with the convention that the estimator is null if the denominator in the previous expression equals 0 and where
\begin{eqnarray*}
    N_j^{x,x'}=\sum_{i=1}^{M_j}\bold{1}_{\{Y_i^{(j)}=x',\,Y_{i-1}^{(j)}=x\}}\quad\text{and}\quad N_j^x=\sum_{i=1}^{M_j}\bold{1}_{\{Y_{i-1}^{(j)}=x\}}.
\end{eqnarray*}
In the next theorem, we provide the theoretical guarantees of our approach and propose a non-asymptotic convergence rate.
\begin{theorem}
\label{theorem::kernelmethod}
Assume that for any $x\in\mathcal{X}$ and $z\in\mathcal{Z}$, the average number of passages of the jump chain at $x$ is greater than a constant $a>0$, i.e. $\inf_{x\in\mathcal{X},z\in\mathcal{Z}}\mathbb{E}_z[N^x]>a$. Also, assume that $\mathbb{P}(S\leq L)<1$ and there exist positive constants $D_L$ and $d_L$ such that  
\begin{eqnarray*}
\sup_{z \in \mathcal{Z}} \mathbb{E}\left[L^l\right] \leq  D_L (l!) (d_L)^l,\quad \forall l \in \mathbb{N}.    
\end{eqnarray*}
Then, for $n$ large enough, one can find a constant $\widetilde C>0$ such that for any $z\in int(\mathcal{Z})$ with $f_Z(z)>0$
    \begin{eqnarray*}
    &&\hspace{-1cm}\mathbb{E}\left[\int_0^{+\infty}\sup_{x\in\mathcal{X}}\left( f_{T,n}^{(k)}(t,x,z)- f_T(t,x,z)\right)^2dt\right]\\
    &&\leq (k+1)^{3/2}\dfrac{\widetilde{C}|\mathcal{X}|}{\lambda_z}\left[\dfrac{1+f_Z(z)}{f_Z(z)^2}h^{2\alpha}+\dfrac{|\mathcal{X}|\sup_{z\in\mathcal{Z}}\mathbb{E}_z[L^2]}{a^2(vf_Z(z)+\psi_n^z)}(nh^p)^{-1}+\exp\left(-nh^p(vf_Z(z)+\psi_n^z)\right)\right]\\
    &&+(k+1)^3\left[\sum_{x\in A}\mathbb{P}(Y_S\neq x,S\leq L)^n+\mathbb{P}(L<S)^n\right]+\dfrac{\lambda_zm_z^2r_z^{2k+1}}{2(1-r_z)}
    \end{eqnarray*}
where $z\mapsto \psi_n^z$ is a measurable function with $\sup_{z\in\mathcal{Z}}|\psi^z_n|=\mathcal{O}(h^\alpha)$ which only depends on $h$ and $f_Z$, $v$ is the volume of the unit sphere with respect to $\Vert .\Vert$ and the constants $m_z>0$ and $0< r_z<1$ are defined in Lemma \ref{lem: expo_decrease}.
\end{theorem}
\noindent
The proof of the above result directly follows from the Proposition \ref{prop::density} with the derivation of the quadratic errors of the aforementioned kernel-based estimators. It particularly shows that the convergence rate of our approach can be almost minimax optimal in the sense of \cite{Cha14}. The selection of series $h$ and $k$ of respective orders $n^{-\frac{1}{2\alpha+p}}$ and $n^\gamma$ with $\gamma>0$ allows the derivation of a density estimator with an average quadratic error of order $n^{\frac{3}{2}\gamma-\frac{2\alpha}{2\alpha+p}}$, which approaches the optimal order of $n^{-\frac{2\alpha}{2\alpha+p}}$ when $\gamma$ decreases. 

\section{Simulation and numerical verifications}  \label{Sec::Simulation}
\noindent
In this section, we illustrate the numerical performances of the kernel based estimator. We propose to define different examples where the random walk could be either irreductible or admit absorbing states. In those examples, we consider covariate variables that take values in $\mathcal{Z}=[0,1]$ and consider the kernel function $ K(u):= \frac{1}{2} \left( \frac{4}{3} - u^2 \right) \mathds{1}_{|u| \leq 1}$. Note that the latter function is inspired from the Epanechnikov's kernel with modified constant, in order to ensure the function bounded away from zero on its support. Although this parameter does not affect the computation, we chose an intensity upper bound of $\widetilde{\lambda}=5$. We consider $N=50$ samples of sizes $n=100,200,400,800$ under the following models
\begin{enumerate}  
    \item[a.] $\mathcal{X} := \{ 1,2,3,4,5,6\}$, $A:=\{5,6 \}$, $Z \sim \mathcal{U}([0,1])$, $L := 5+X$ with $X \sim \mathcal{P}(2)$, $\lambda_z := 1 + z$ and 
    \begin{eqnarray*}
        P^{(z)} :=  \begin{pmatrix}
    0 & \frac{z}{1 + z^2} & 0 & \frac{1 + z^2 - z}{1 + z^2} & 0 & 0 \\[.2cm]
    0 & \frac{0.5}{1.1 + z + z^2} & \frac{0.4 + z}{1.1 + z + z^2} & 0 & \frac{0.1 + z^2}{1.1 + z + z^2} & 0 \\[.2cm]
    0.4 &  \frac{0.3}{2 + z} &  \frac{0.6 + 0.6z }{2 + z} & 0 & 0 & \frac{0.3}{2 + z} \\[.2cm]
    0.3 & 0.5 & 0 & 0 & 0.2 & 0 \\[.2cm]
    0 & 0 & 0 & \frac{1+z}{3+z} & 0 & \frac{2}{3+z} \\[.2cm]
    0 & 0 & 0.3 & 0 & 0.5 & 0.2 
    \end{pmatrix} .
    \end{eqnarray*}
    \item[b.] $\mathcal{X} := \{ 1,2,3,4,5,6,7 \}$, $A:=\{6,7 \}$, $Z \sim \mathcal{B}eta(1.4,2.7)$, $L := 6+X$ with $X \sim \mathcal{P}(1)$, $\lambda_z := 0.4
    +2z^{3/4}$ and 
    \begin{eqnarray*}
        P^{(z)} :=  \begin{pmatrix}
    0 & \frac{1.5 + z}{1 + z + z^3} & \frac{0.5 + 0.3z^3}{1 + z + z^3} & 0 & 0 & \frac{0.7z^3}{1 + z + z^3} & 0 \\[.2cm]
    \frac{0.5}{(1 + z)^2} & 0 & 0 & \frac{0.3 + 1.5z}{(1 + z)^2} & \frac{0.5z + 0.5z^2}{(1 + z)^2} & 0 & \frac{0.2 + 0.5 z^2}{(1 + z)^2} \\[.2cm]
    0 & 0 &  0.2 &  \frac{0.4}{1 + z} & 0 & \frac{0.4 + 0.8z}{1 + z} & 0 \\[.2cm]
    0.1 + 0.3\sqrt{z} & 0 & 0.3 + 0.1 \sqrt{z} & 0 & 0 & 0 & 0.6 - 0.4 \sqrt{z} \\[.2cm]
    0 & 0 & 0 & 0 & 1 & 0 & 0 \\[.2cm]
    0 & 0.4 & 0 & 0 & 0  & 0 & 0 \\[.2cm]
    0 & 0 & 0 & 0 & 0.2  & 0.8 & 0
    \end{pmatrix}
    \end{eqnarray*}
\end{enumerate}
where $\mathcal{U}, \mathcal{P}$ and $\mathcal{B}eta$ respectively denote the uniform, the Poisson and the Beta distributions. The first example illustrates a case where the Markov chain is irreductible and the model functions are Lipchitz. In the second example, the fifth state is settled to be an absorbing state and the model functions are Hölder continuous with $\alpha = 0.4$. \\

\noindent
In order to obtain a qualitative measure the estimator consistency, we propose to draw the boxplots of the risk error for some fixed values of $z$, i.e.
$$ \left\{ \mathcal{I}^{m}(n,k,z) := \int_0^{+\infty}\sup_{x\in\mathcal{X}}\left( f_{T,n}^{(k)(m)}(t,x,z)- f_T(t,x,z)\right)^2dt , \, m = 1, \ldots , N  \right\}$$
where here and throughout this section, the index $m$ designates the sample number given to an estimator. Proposition \ref{prop::density} provides upper bounds for the risk error, which relies primarily on the estimation of the intensity parameter and of the coefficients $(c_j(x,z))_{j \in \mathbb{N}}$. We therefore determine which of these estimators induces the greatest volatility and leads to a larger risk error amplitude, drawing the boxplots of the estimators
$$\left\{ \frac{\lambda_{n,z}^{(m)}}{\lambda_{z}},\,m =1, \ldots , N \right\} \quad \text{and} \quad \left\{\sup_{x\in\mathcal{X},\,j=1,\ldots,k}\left(c^{(k)(m)}_{n,j}(x,z)-c_j(x,z)\right)^2,\,m =1, \ldots , N \right\}.  $$
It is worth-noting that the true density function is represented as an infinite sum of preponderated Erlang distributions. Only partial sums of this function can consequently be numerically computed. However, the machine error for values of smaller order than $10^{-16}$ do not allow to consider partial sums of order greater that 135. We consequently restrain the approximation of the true density function by up to $k=130$ iterations for all the estimation procedures. \\

The choice of the parameter $h$ can be cumbersome in numerical computations of kernel-based estimators. In our experiments, it turned out that poor estimations of the intensity parameters $\lambda_z$ were causing a greater impact on the risk error value than the estimation of the coefficients $\{ c_j(x,z) , x \in \mathcal{X} , z \in \mathcal{Z} \}.$ Therefore, we propose basing the choice of bandwidth on minimizing the volatility of the intensity estimation, which is given as follows. Given $z \in \mathcal{Z}$, we define the estimator of the multiplicative inverse of the intensity parameter $\lambda_z^{-1}$ as $$ \widehat{m}_h(z) 
:= \frac{1}{\widehat{\lambda}_h(z)} = \frac{\sum_{i=1}^n K_h(Z_i - z)\, \frac{E^{(i)}_{M_i}}{M_i}}
       {\sum_{i=1}^n K_h(Z_i - z)} $$ 
so that for any $j \in \{ 1 , \ldots , n \}$, $\widehat{m}_{h,-j}(z)$ refers to the above estimator when the $j$-th element of the sample is removed. Bandwidth candidates are then selected by minimizing the following conditional predictive error
\begin{eqnarray*}
    CPE(h) = \sum_{j=1}^n \left( \frac{E^{(j)}_{M_j}}{M_j} - \widehat{m}_{h,-j}(z)  \right)^2 .
\end{eqnarray*}
The chosen bandwidth is selected using a tenfold sampling procedure. Formally, we consider ten sub-samples $(A_{n,\ell})_{\ell=1,\ldots,10}$, where $A_{n,\ell}$ regroups all but the elements of indexes in $\{ 10(\ell-1) +1 , \ldots , 10 \ell \}$. This allows to introduce $h_\ell$ as the parameter minimizing the conditional predictive error for the sample $A_{n,\ell}$. The selected bandwidth then represents the average value of the aforementioned candidates with $h := \frac{1}{10} \sum_{\ell =1}^{10} h_{\ell}$.\\

\noindent
Boxplots from Figures \ref{fig::risk_dens_a}-\ref{fig::risk_coeff_b} illustrate the performances of the kernel based estimator. For any boxplot illustration and any covariate value, positions one through four designate the boxes from left to right, respectively

\begin{figure}[H]
    \centering
    \includegraphics[scale=0.4]{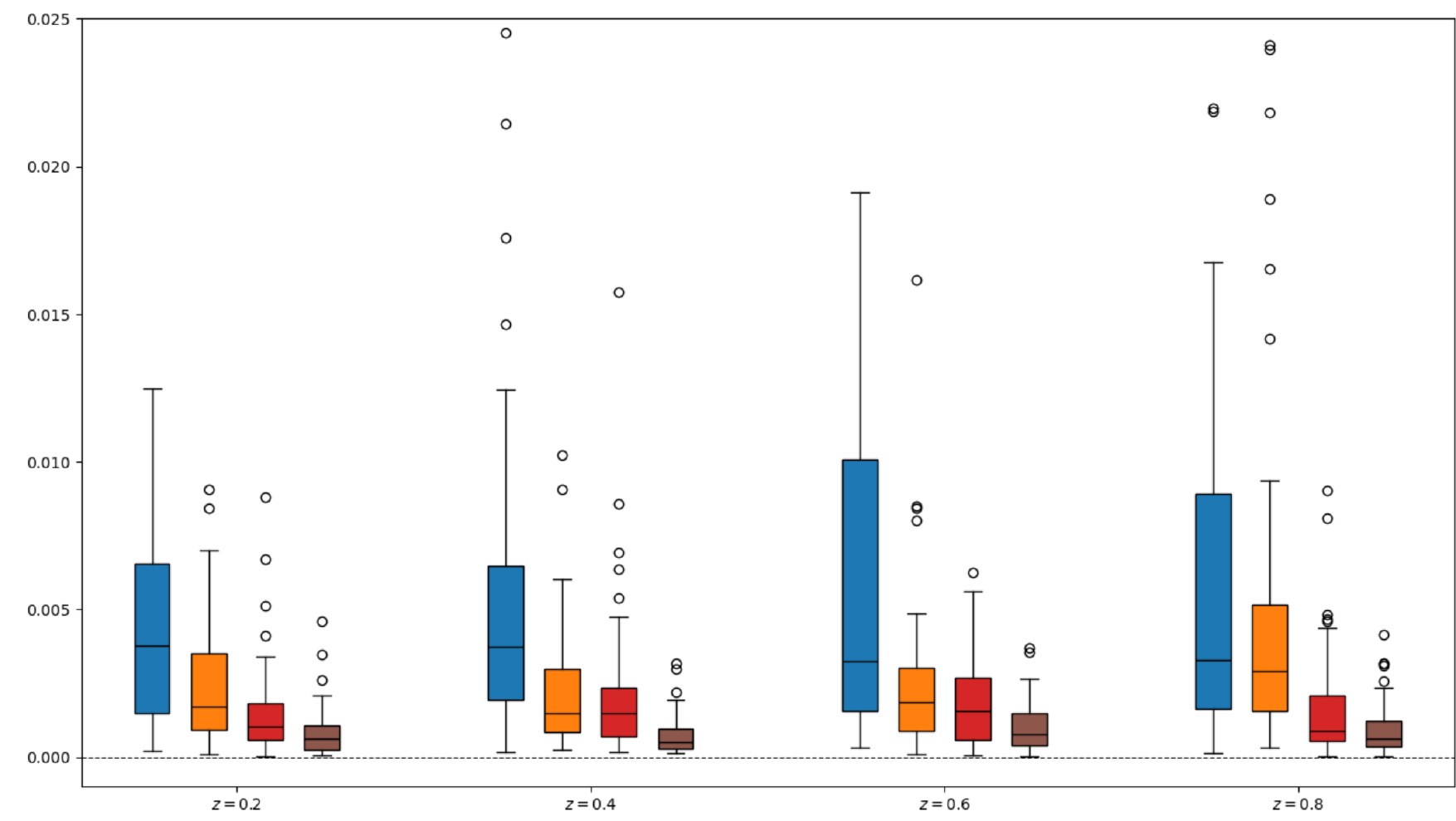}
    \caption{Model a - Risk error of the survival density estimator for $z=0.2,0.4,0.6,0.8$, for $n=100$ (position 1), 200 (position 2), 400 (position 3) and 800 (position 4).}
    \label{fig::risk_dens_a}
\end{figure}

\begin{figure}[H]
    \centering
    \includegraphics[scale=0.4]{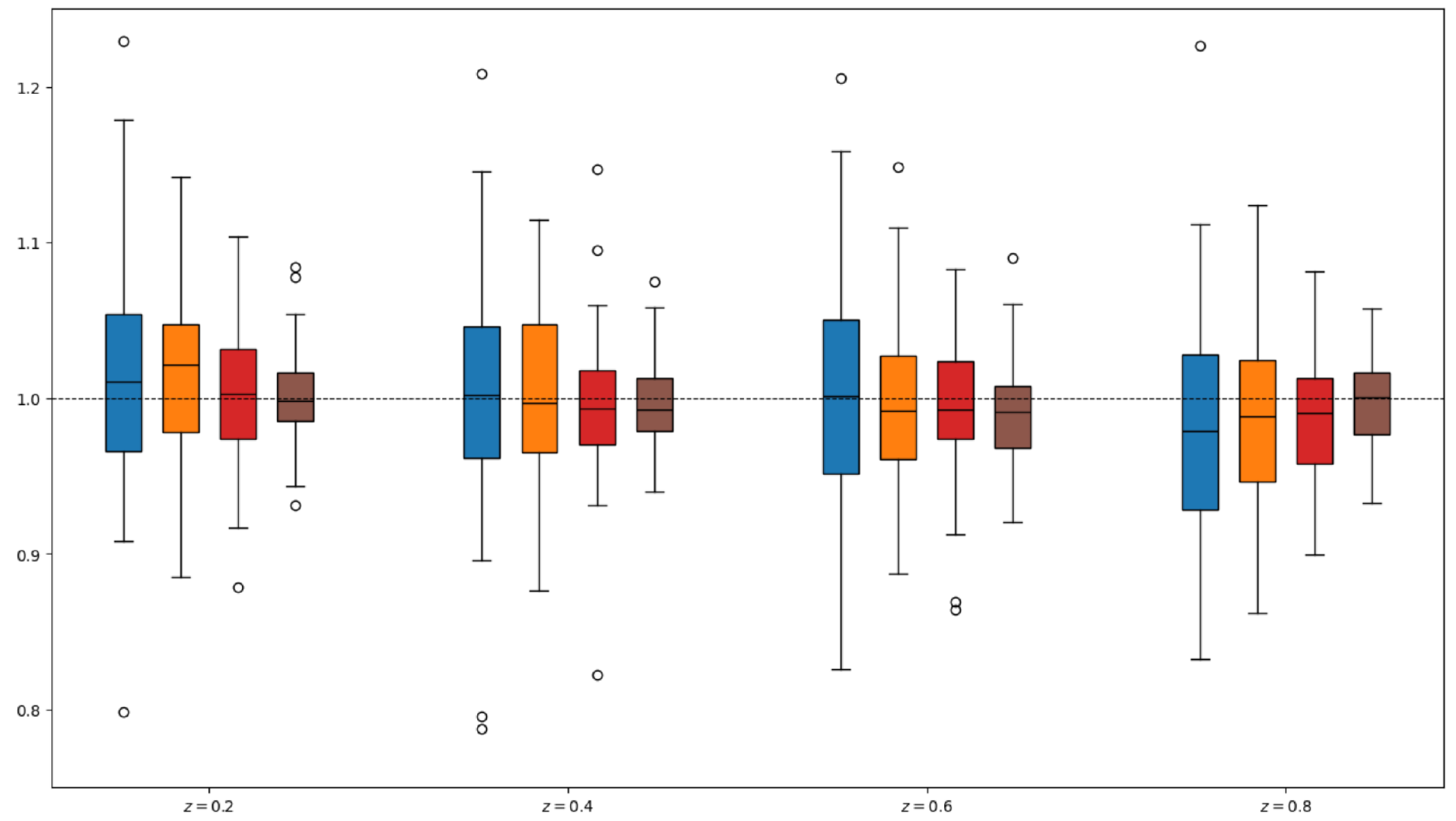}
    \caption{Model a - Normalized estimator of the intensity parameter for $z=0.2,0.4,0.6,0.8$, for $n=100$ (position 1), 200 (position 2), 400 (position 3) and 800 (position 4).}
    \label{fig::int_a}
\end{figure}

\begin{figure}[H]
    \centering
    \includegraphics[scale=0.4]{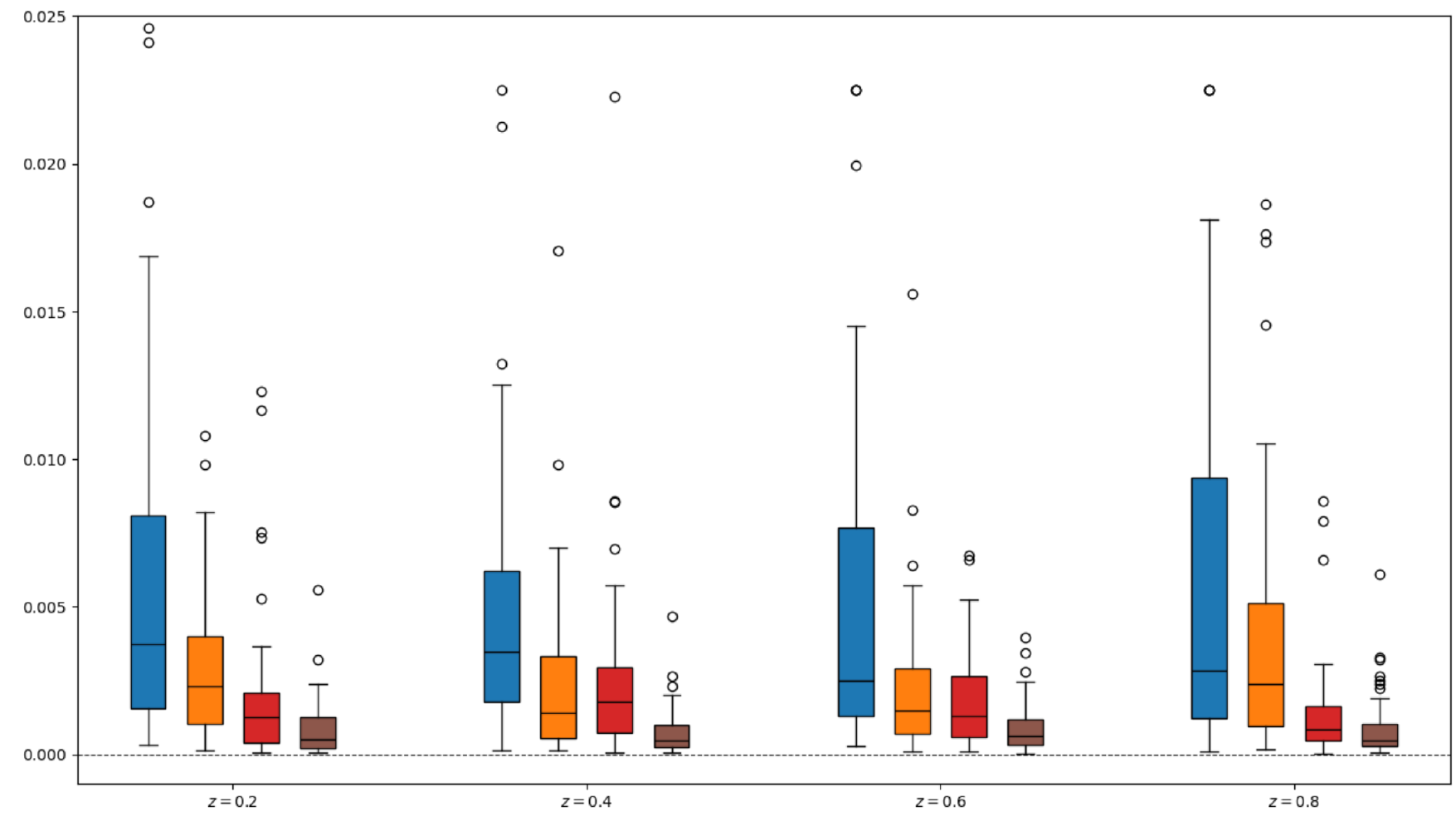}
    \caption{Model a - Risk error of the maximum of the survival coefficient for $z=0.2,0.4,0.6,0.8$, for $n=100$ (position 1), 200 (position 2), 400 (position 3) and 800 (position 4).}
    \label{fig::risk_coeff_a}
\end{figure}

\begin{figure}[H]
    \centering
    \includegraphics[scale=0.4]{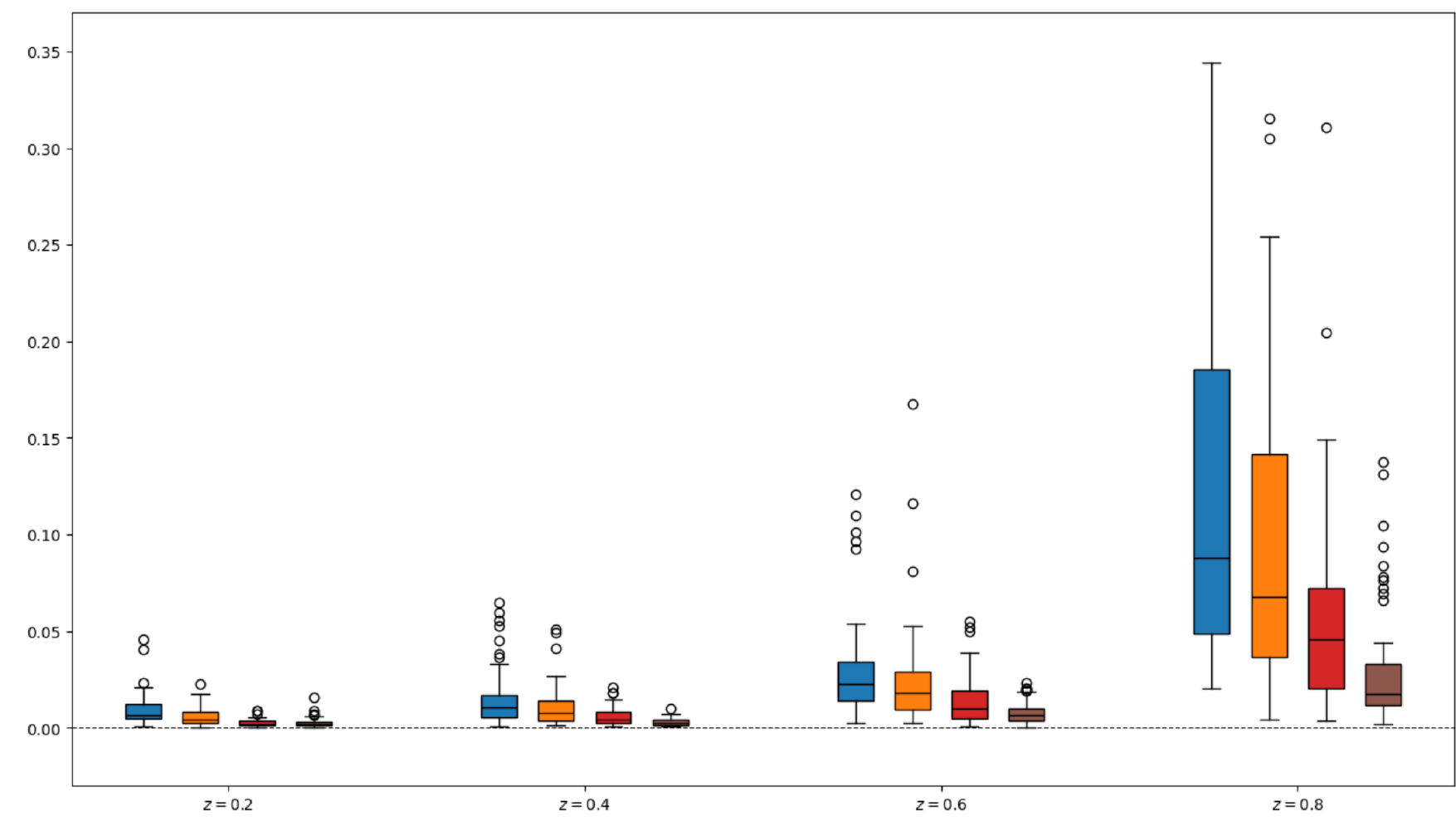}
    \caption{Model b - Risk error of the survival density estimator for $z=0.2,0.4,0.6,0.8$, for $n=100$ (position 1), 200 (position 2), 400 (position 3) and 800 (position 4).}
    \label{fig::risk_dens_b}
\end{figure}

\begin{figure}[H]
    \centering
    \includegraphics[scale=0.4]{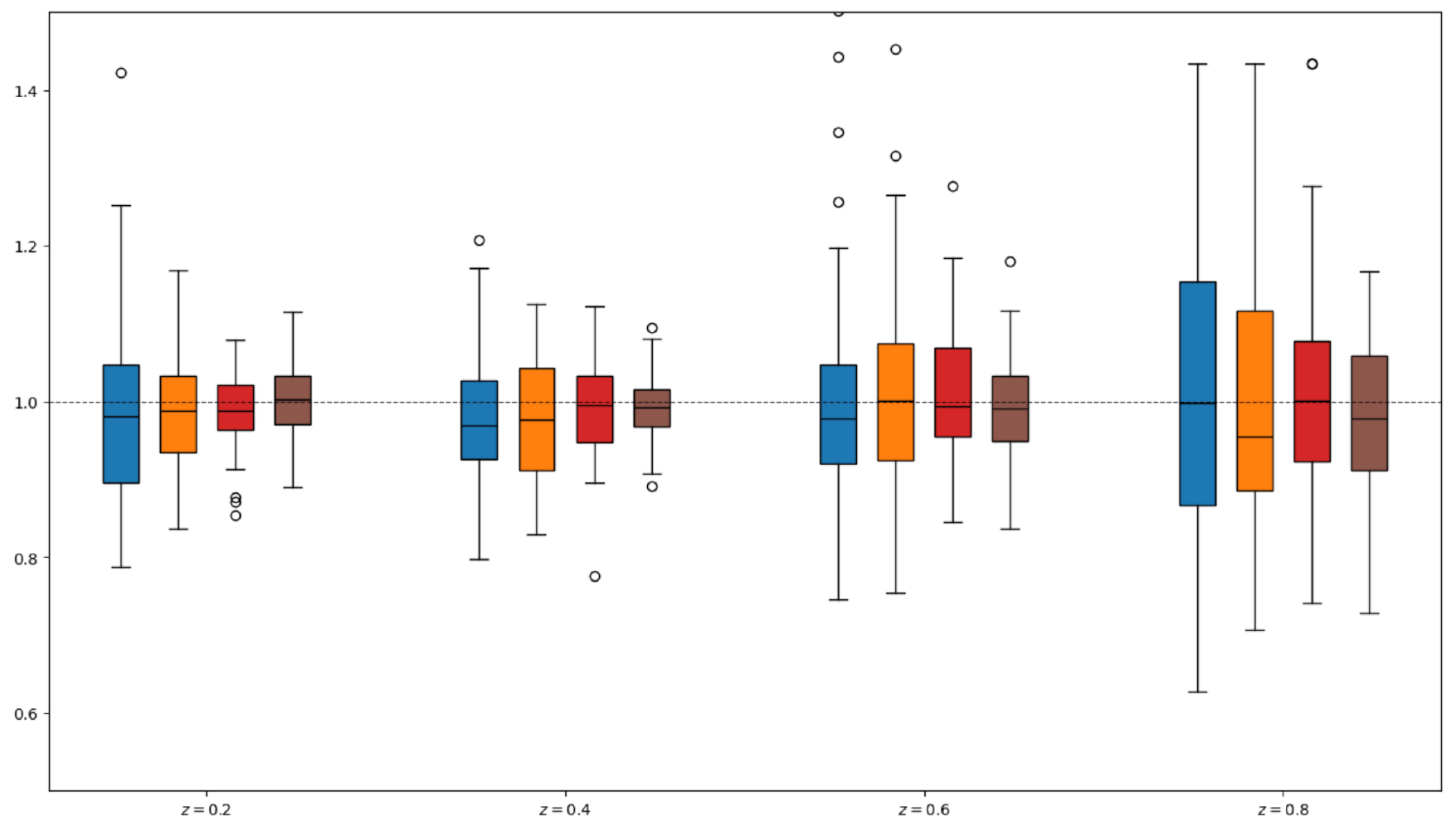}
    \caption{Model b - Normalized estimator of the intensity parameter for $z=0.2,0.4,0.6,0.8$, for $n=100$ (position 1), 200 (position 2), 400 (position 3) and 800 (position 4).}
    \label{fig::int_b}
\end{figure}

\begin{figure}[H]
    \centering
    \includegraphics[scale=0.4]{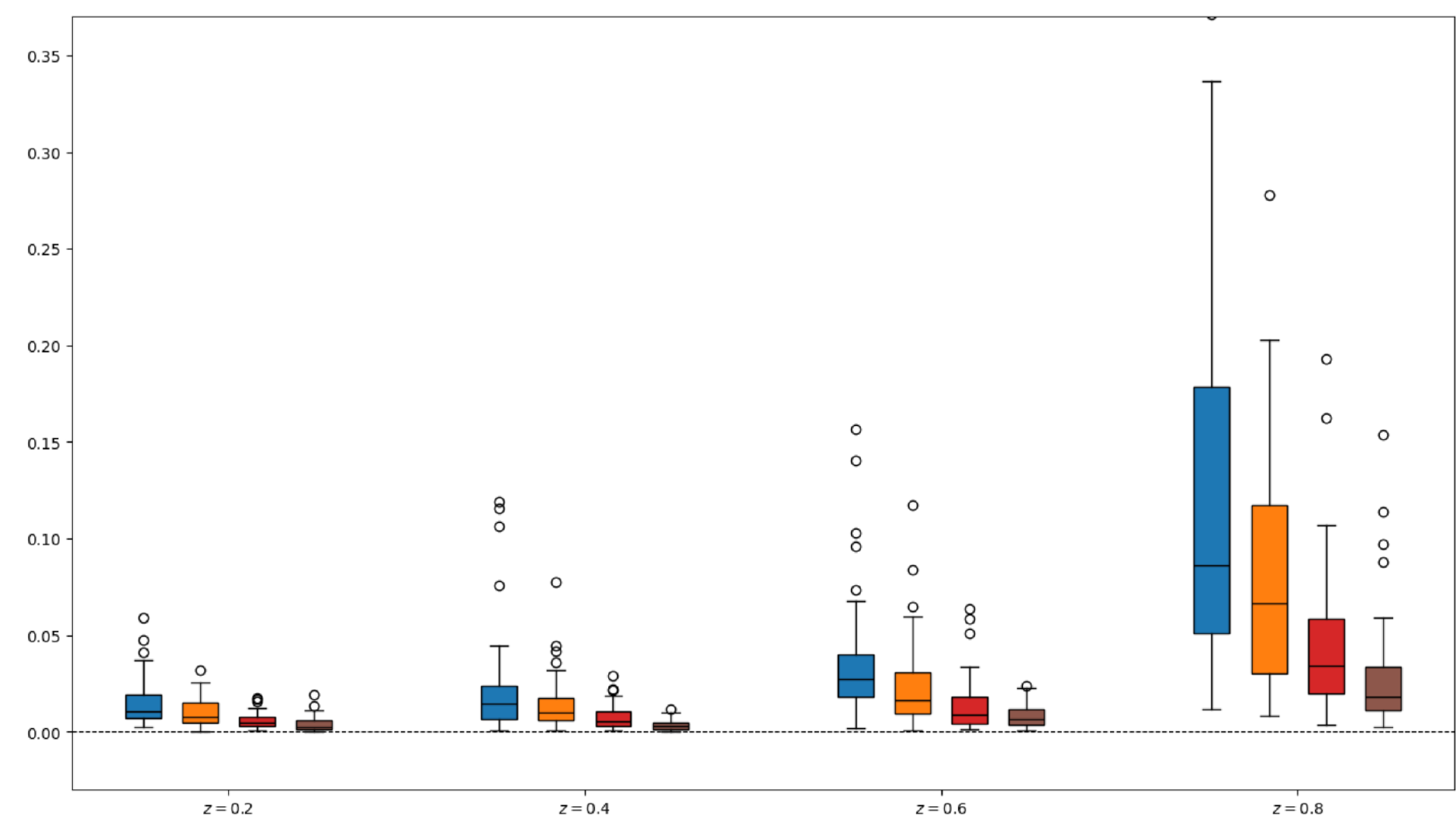}
    \caption{Model b - Risk error of the maximum of the survival coefficient for $z=0.2,0.4,0.6,0.8$, for $n=100$ (position 1), 200 (position 2), 400 (position 3) and 800 (position 4).}
    \label{fig::risk_coeff_b}
\end{figure}

\noindent
The results obtained for each of the models differ significantly. Although the estimator of the first model performs well, even for small sample sizes and independently of the covariate, the second model illustrates a greater difficulty in estimating the intensity parameter and the true density for values of $z$ close to 1. As explained at the beginning of this section, the density function of $Z$ offers less regularity in the second case than in the first one, suggesting slower rates of convergence. Moreover, it is worth mentioning that the $0.98$ (resp. $0.95$) quantile is approximately equal to $0.808$ (resp. $0.728$) when $Z \sim \mathcal{B}eta(1.4,2.7)$. As expected, it was more difficult to estimate the true density function in the second model for values of z near 1, since kernel-based estimation methods typically perform poorly in the edges of the support, where observations are usually sparse. Finally, whether model a. or model b. is considered, the risk error values decrease significantly as the sample size increases.

\section{Bone marrow transplant}
In this section, we illustrate our method on a dataset compiled from the Center for International Blood and Marrow Transplant Research and studied in \cite{And08}. The patient cohort includes $n=2009$ individuals who have received a marrow transplant between 1995 and 2004 from a sibling for acute myelogenous leukemia (AML) or acute lymphoblastic leukemia (ALL). Follow-up with each patient allows us to construct their clinical trajectory after bone marrow or blood stem cell transplantation (state 0). Among patients who experienced death (state 1), three main competing events are considered: relapse (state 2), donor rejection of the transplant (state 3) and the moment when the absolute neutrophil count (ANC) exceeds above 500 cells per $\mu L$ (state 4), representing a higher risk of infection. Figure \ref{fig::states} resumes the configuration of the states with the number of transitions and Table \ref{tab::data} presents the covariates we considered for each patient.

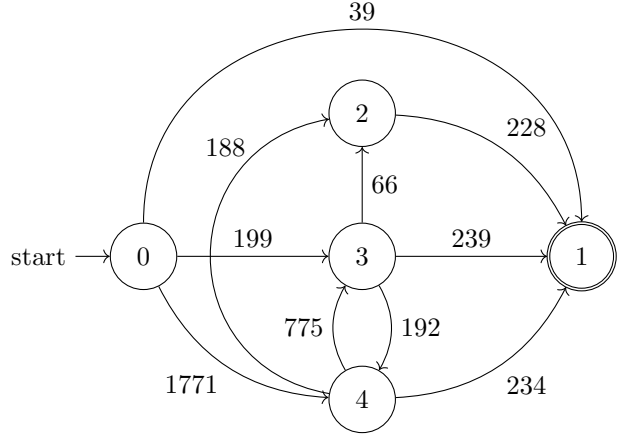
\begin{figure}[h]
\centering
\begin{tikzpicture}
    \node[state] (3) {3};
    \node[state, initial] (0) [left=2cm of 3] {0};
    \node[state, above=of 3] (2) {2};
    \node[state, below =of 3] (4) {4};
    \node[state, accepting, right=2cm of 3](1) {1};

\path[->] 
    (0) edge node[above] {199} (3)
    (0) edge[bend right] node[below left] {1771} (4)
    (0) edge[bend left=90,looseness=1.5] node[above] {39} (1)
    (2) edge[bend left] node[above right] {228} (1)
    (3) edge node[right] {66} (2)
    (3) edge node[above] {239} (1)
    (3) edge[bend left] node[right] {192} (4)
    (4) edge[bend left] node[left] {775} (3)
    (4) edge[bend left=80, looseness=1.5] node[above, xshift=.2cm, yshift=1.2cm] {188} (2)
    (4) edge[bend right] node[below right] {234} (1);    
\end{tikzpicture}
\caption{States and number of transitions.}
\label{fig::states}
\end{figure}

\begin{table}[h]
\centering
\begin{tabular}{r l}
    \texttt{sex}: & 0 = female, 1 = male.\\
    \texttt{age}: & age in years.\\
    \texttt{all}: & 0 = AML, 1 = ALL.\\
    \texttt{bmonly}: & 0 = peripheral blood/bone marrow, 1 = only bone marrow.
\end{tabular}
\caption{Covariate description.}
\label{tab::data}
\end{table}

\noindent
Our approach for this application is based on a twofold likelihood estimation of the transition probabilities and the intensities. Because the covariates are a combination of categorical and quantitative information, we consider neural networks as proper candidates for the conditional functions. Each network is based on 10 hidden layers with 1024 nodes, ReLU activation function and optimized throughout a set of parameter weights $\omega\in\Omega$. The networks intended for the transition probabilities are denoted by $L_{\omega}$ and map an initial position $x$ and a covariate $z\in\mathcal{Z}$ to the unit simplex $\Delta^{|\mathcal{X}|}$. Similarly, networks denoted $\widetilde{L}_\omega$ are designed for the intensity function, take values in $\mathcal{Z}$ and return a positive real number. The sets of the neural network parameters are trained by gradient descent using the Adam optimizer \citep{kingma2014adam} for $n_{\textit{iter}} = 1000$ epochs to maximize the following log-likelihood 
\begin{eqnarray*}
    \mathcal{L}_p(\omega)=\sum_{j=1}^n\sum_{i=1}^{M_j}\log\left(L_\omega(Y_{i-1}^{(j)},Z_j)_{Y_i^{(j)}}\right)
\quad \text{and}\quad\mathcal{L}_\lambda(\omega)=\sum_{j=1}^nM_j\log\left(\widetilde{L}_\omega(Z_j)\right)-\widetilde{L}_\omega(Z_j)E^{(j)}_{M_j}
\end{eqnarray*}
and obtain the estimators $p_{n,z}(x,x')=L_{\omega_1}(x,z)_{x'}$ and $\lambda_{n,z}=\widetilde{L}_{\omega_2}(z)$ where $\omega_1=\arg\max_{\omega\in\Omega} \mathcal{L}_p(\omega)$ and $\quad\omega_2=\arg\max_{\omega\in\Omega} \mathcal{L}_\lambda(\omega)$. Our procedure is then applied with $k=25$ iterations, resulting in density and survival functions for all covariate configurations where part of the results are presented in Figure \ref{fig::curves1} and \ref{fig::curves2}.\\ 

\noindent
Overall, it appears that death occurs more quickly after surgery than relapse, which is expected based on the path configuration of the state space since relapse can only result in death. However, comparison of the blue curves suggests that younger patients are likely to survive for longer. Furthermore, there appears to be no discrepancy between young patients who have experienced transplant rejection or signs of infection. On the contrary, old patients suffering transplant rejection are more likely to die faster than those experiencing infection.

\begin{figure}[H]
    \centering
    \includegraphics[width=0.7\linewidth,page=13]{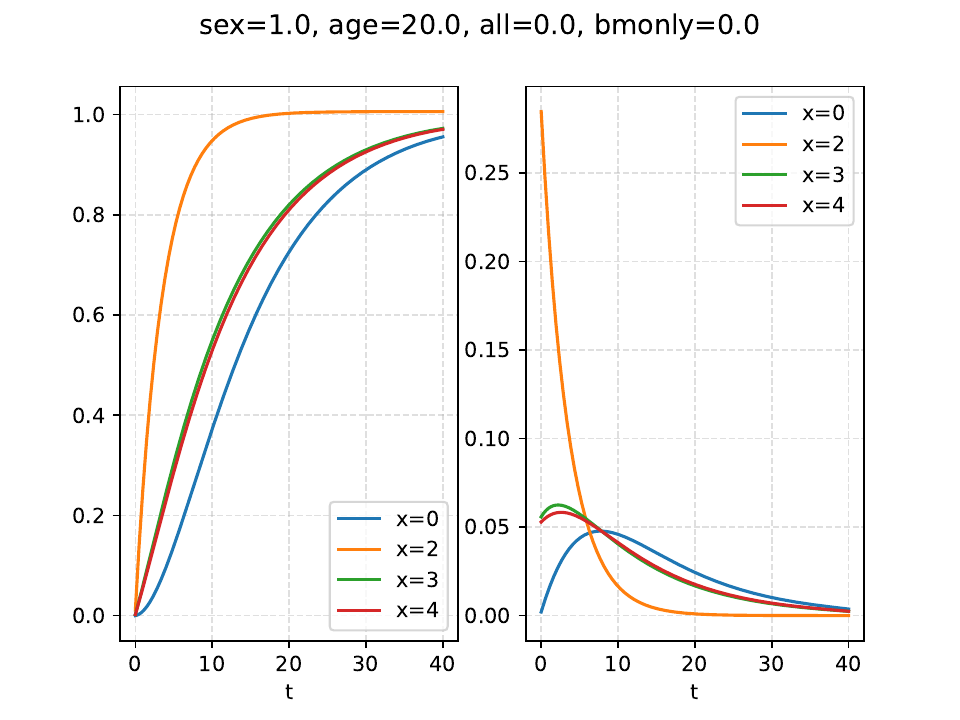}
    \caption{Survival and density estimations as functions of time (months) for male patients with 20 years old, acute myelogenous leukemia and peripheral blood/bone marrow. }
    \label{fig::curves1}
\end{figure}
\begin{figure}[H]
    \centering
    \includegraphics[width=0.7\linewidth,page=21]{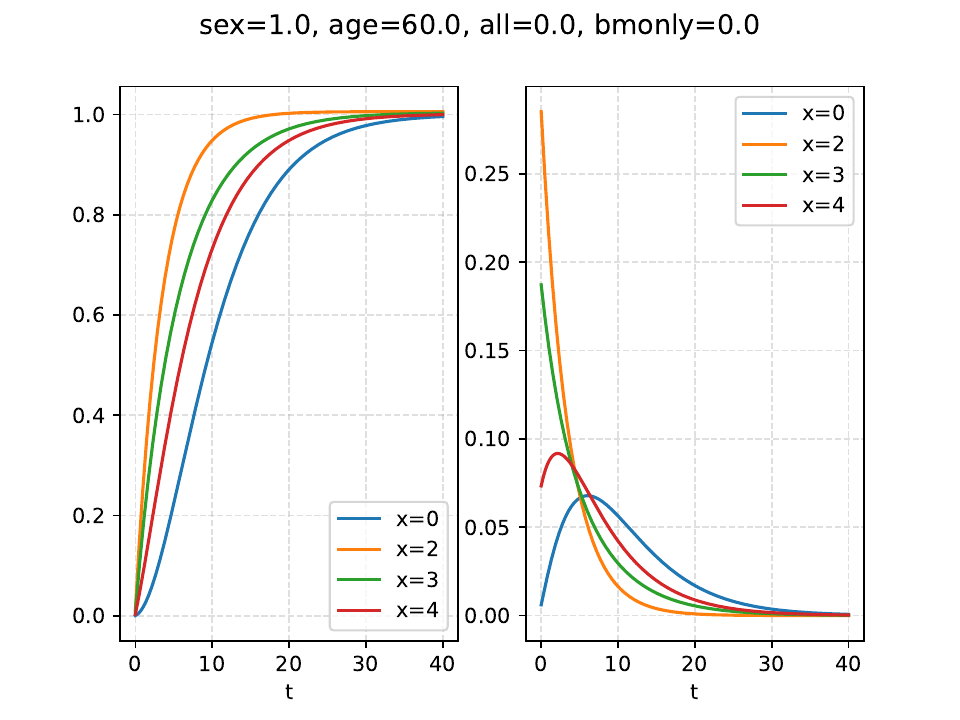}
    \caption{Survival and density estimations as functions of time (months) for male patients with 60 years old, acute myelogenous leukemia and peripheral blood/bone marrow. }
    \label{fig::curves2}
\end{figure}

\bibliographystyle{abbrv}
\bibliography{bibli}

\begin{thebibliography}{10}

\bibitem{Aal78}
O.~O. Aalen and S.~Johansen.
\newblock An empirical transition matrix for non-homogeneous {M}arkov chains
  based on censored observations.
\newblock {\em Scandinavian Journal of Statistics}, 5(3):141--150, 1978.

\bibitem{Ami18}
M.~Amico and I.~Van~Keilegom.
\newblock Cure models in survival analysis.
\newblock {\em Annu. Rev. Stat. Appl.}, 5(1):311--342, Mar. 2018.

\bibitem{And02}
P.~K. Andersen, S.~Z. Abildstrom, and S.~Rosthøj.
\newblock Competing risks as a multi-state model.
\newblock volume~11, pages 203--215, 2002.

\bibitem{And02b}
P.~K. Andersen and N.~Keiding.
\newblock Multi-state models for event history analysis.
\newblock volume~11, pages 91--115, 2002.

\bibitem{And08}
P.~K. Andersen and M.~Pohar~Perme.
\newblock Inference for outcome probabilities in multi-state models.
\newblock {\em Lifetime Data Anal.}, 14(4):405--431, 2008.

\bibitem{And22}
P.~K. Andersen, E.~N.~S. Wandall, and M.~Pohar~Perme.
\newblock Inference for transition probabilities in non-{M}arkov multi-state
  models.
\newblock {\em Lifetime Data Analysis}, 28(4):585--604, Oct 2022.

\bibitem{Asa15}
{\"O}.~Asar, J.~Ritchie, P.~A. Kalra, and P.~J. Diggle.
\newblock Joint modelling of repeated measurement and time-to-event data: an
  introductory tutorial.
\newblock {\em Int. J. Epidemiol.}, 44(1):334--344, Feb. 2015.

\bibitem{Mas98}
L.~Birgé and P.~Massart.
\newblock Minimum contrast estimators on sieves: Exponential bounds and rates
  of convergence.
\newblock {\em Bernoulli}, 4(3):329--375, 1998.

\bibitem{Cha14}
G.~Chagny and A.~Roche.
\newblock Adaptive and minimax estimation of the cumulative distribution
  function given a functional covariate.
\newblock {\em Electron. J. Stat.}, 8(2):2352--2404, 2014.

\bibitem{Chi68}
C.~L. Chiang.
\newblock {\em Introduction to Stochastic Processes in Biostatistics},
  volume~12.
\newblock 1968.

\bibitem{Una15}
J.~de~Uña-Álvarez and L.~Meira-Machado.
\newblock Nonparametric estimation of transition probabilities in the
  non-{M}arkov illness-death model: A comparative study.
\newblock {\em Biometrics}, 71(2):364--375, 2015.

\bibitem{Egea25}
M.~Eg\'ea and M.~Escobar-Bach.
\newblock Local differential privacy in survival analysis using private failure
  indicators.
\newblock {\em Electron. J. Stat.}, 19(1):2603--2636, 2025.

\bibitem{Fry95}
H.~Frydman.
\newblock Nonparametric estimation of a {M}arkov `illness-death' process from
  interval- censored observations, with application to diabetes survival data.
\newblock {\em Biometrika}, 82(4):773--789, 1995.

\bibitem{Hou99}
P.~Hougaard.
\newblock Multi-state models: A review.
\newblock {\em Lifetime Data Analysis}, 5(3):239--264, Sep 1999.

\bibitem{Hou00}
P.~Hougaard.
\newblock {\em Analysis of Multivariate Survival Data}, volume 542.
\newblock 01 2000.

\bibitem{Jol02}
P.~Joly, D.~Commenges, C.~Helmer, and L.~Letenneur.
\newblock A penalized likelihood approach for an illness-death model with
  interval-censored data: application to age-specific incidence of dementia.
\newblock {\em Biostatistics}, 3(3):433--443, Sept. 2002.

\bibitem{Kay86}
R.~Kay.
\newblock A {M}arkov model for analysing cancer markers and disease states in
  survival studies.
\newblock {\em Biometrics}, 42(4):855--865, 1986.

\bibitem{Ken25}
E.~B. Kendall, J.~P. Williams, G.~H. Hermansen, F.~Bois, and V.~H. Thanh.
\newblock Beyond time-homogeneity for continuous-time multistate {M}arkov
  models.
\newblock {\em Journal of Computational and Graphical Statistics},
  34(2):668--682, 2025.

\bibitem{kingma2014adam}
D.~P. Kingma and J.~Ba.
\newblock Adam: A method for stochastic optimization.
\newblock In {\em International Conference on Learning Representations (ICLR)},
  2015.

\bibitem{Li23}
W.~Li, L.~Li, and B.~C. Astor.
\newblock A comparison of two approaches to dynamic prediction: Joint modeling
  and landmark modeling.
\newblock {\em Stat. Med.}, 42(13):2101--2115, June 2023.

\bibitem{Lin22}
J.~Lin and S.~Luo.
\newblock Deep learning for the dynamic prediction of multivariate longitudinal
  and survival data.
\newblock {\em Stat. Med.}, 41(15):2894--2907, July 2022.

\bibitem{Lou17}
W.~Lou, L.~Wan, E.~L. Abner, D.~W. Fardo, H.~H. Dodge, and R.~J. Kryscio.
\newblock Multi-state models and missing covariate data:
  {Expectation-Maximization} algorithm for likelihood estimation.
\newblock {\em Biostat. Epidemiol.}, 1(1):20--35, Apr. 2017.

\bibitem{Mac21}
R.~J. Machado, A.~{van den Hout}, and G.~Marra.
\newblock Penalised maximum likelihood estimation in multi-state models for
  interval-censored data.
\newblock {\em Computational Statistics \& Data Analysis}, 153:107057, 2021.

\bibitem{Mal21}
N.~Maltzahn, R.~Hoff, O.~O. Aalen, I.~S. Mehlum, H.~Putter, and J.~M. Gran.
\newblock A hybrid landmark {A}alen-{J}ohansen estimator for transition
  probabilities in partially non-{M}arkov multi-state models.
\newblock {\em Lifetime Data Analysis}, 27(4):737--760, Oct 2021.

\bibitem{May22}
L.~Maystre and D.~Russo.
\newblock Temporally-consistent survival analysis.
\newblock In S.~Koyejo, S.~Mohamed, A.~Agarwal, D.~Belgrave, K.~Cho, and A.~Oh,
  editors, {\em Advances in Neural Information Processing Systems}, volume~35,
  pages 10671--10683. Curran Associates, Inc., 2022.

\bibitem{Mei06}
L.~Meira-Machado, J.~de~U{\~{n}}a-{\'A}lvarez, and C.~Cadarso-Su{\'a}rez.
\newblock Nonparametric estimation of transition probabilities in a
  non-{M}arkov illness--death model.
\newblock {\em Lifetime Data Analysis}, 12(3):325--344, Sep 2006.

\bibitem{Mei09}
L.~Meira-Machado, J.~de~U{\~n}a-Alvarez, C.~Cadarso-Su{\'a}rez, and P.~K.
  Andersen.
\newblock Multi-state models for the analysis of time-to-event data.
\newblock {\em Stat. Methods Med. Res.}, 18(2):195--222, Apr. 2009.

\bibitem{Pep91}
M.~S. Pepe.
\newblock Inference for events with dependent risks in multiple endpoint
  studies.
\newblock {\em Journal of the American Statistical Association},
  86(415):770--778, 1991.

\bibitem{Per01}
R.~P{\'e}rez-Oc{\'o}n, J.~E. Ruiz-Castro, and M.~L. G{\'a}miz-P{\'e}rez.
\newblock A piecewise {M}arkov process for analysing survival from breast
  cancer in different risk groups.
\newblock {\em Statistics in Medcine}, 20(1):109--122, January 2001.

\bibitem{Put02}
H.~Putter, M.~Fiocco, and R.~B. Geskus.
\newblock Tutorial in biostatistics: competing risks and multi-state models.
\newblock volume~26, pages 2389--2430, 2007.

\bibitem{Put17}
H.~Putter and H.~C. van Houwelingen.
\newblock Understanding landmarking and its relation with time-dependent cox
  regression.
\newblock {\em Stat. Biosci.}, 9(2):489--503, 2017.

\bibitem{Put22}
H.~Putter and H.~C. van Houwelingen.
\newblock Landmarking 2.0: Bridging the gap between joint models and
  landmarking.
\newblock {\em Stat. Med.}, 41(11):1901--1917, May 2022.

\bibitem{Riz11}
D.~Rizopoulos.
\newblock Dynamic predictions and prospective accuracy in joint models for
  longitudinal and time-to-event data.
\newblock {\em Biometrics}, 67(3):819--829, 2011.

\bibitem{Riz17}
D.~Rizopoulos, G.~Molenberghs, and E.~M. E.~H. Lesaffre.
\newblock Dynamic predictions with time-dependent covariates in survival
  analysis using joint modeling and landmarking.
\newblock {\em Biom J}, 59(6):1261--1276, Aug. 2017.

\bibitem{Str98}
D.~Strauss and R.~Shavelle.
\newblock An extended {K}aplan–{M}eier estimator and its applications.
\newblock {\em Statistics in Medicine}, 17(9):971--982, 1998.

\bibitem{Sur17}
K.~Suresh, J.~M.~G. Taylor, D.~E. Spratt, S.~Daignault, and A.~Tsodikov.
\newblock Comparison of joint modeling and landmarking for dynamic prediction
  under an illness-death model.
\newblock {\em Biom. J.}, 59(6):1277--1300, Nov. 2017.

\bibitem{Tit11}
A.~C. Titman.
\newblock Flexible nonhomogeneous {M}arkov models for panel observed data.
\newblock {\em Biometrics}, 67(3):780--787, 2011.

\bibitem{Tsi04}
A.~A. Tsiatis and M.~Davidian.
\newblock Joint modeling of longitudinal and time-to-event data: An overview.
\newblock {\em Statistica Sinica}, 14(3):809--834, 2004.

\bibitem{Hou07}
H.~C. van Houwelingen.
\newblock Dynamic prediction by landmarking in event history analysis.
\newblock {\em Scandinavian Journal of Statistics}, 34(1):70--85, 2007.

\bibitem{Tsi97}
M.~S. Wulfsohn and A.~A. Tsiatis.
\newblock A joint model for survival and longitudinal data measured with error.
\newblock {\em Biometrics}, 53(1):330--339, 1997.

\end{thebibliography}

\section{Proof}  \label{Sec::Proof}

\subsection{Proof of Lemma \ref{lem: expo_decrease}}
This proof is based on the set $B_z$ defined in subsection \ref{subsec:: absorbcure}. The proof is direct if $x\notin B_z$:  for any $j\in\mathbb{N}$, if $x\in A$ then $c_j(x,z)=\bold{1}_{\{j=0\}}$, and if $x\in C_z$ then $c_j(x,z)=0$. We thus only consider starting position $x$ with finite path to the subset $A\cup C_z$. If $x\in B_z$, we denote $\ell_z(x)$ the smallest path length from $x$ to the subset $A\cup C_z$. Since $\mathcal{X}$ is finite, we have $d_z := \max_{x \in \mathcal{X}} \ell_z (x)<+\infty$ so that
\begin{eqnarray*}
    \varepsilon_z := \min_{ x \in B_z } \mathbb{P}_z \left( \exists j \in \{1, \ldots , d_z \} , Y_j \in A\cup C_z \;  | \; Y_0 = x \right)>0.
\end{eqnarray*}

\noindent
We introduce $S'$ as the first hitting time of the chain to the set $C_z$, so that $\widetilde{S}:=S\wedge S'$ is a stopping time returning the first time the chain reaches $A\cup C_z$. Because the state space is finite, we ensure that $S\wedge S'<+\infty$ a.s and we prove by induction that $\forall x \in B_z$ and $k \in \mathbb{N}$
$$ \mathbb{P}_z \left( \widetilde{S}> kd_z \;  | \;   Y_0 = x   \right) \leq (1 - \varepsilon_z)^k .$$
The above inequality is trivial for $k=1$ and only remain to prove the induction step. We assume the inequality holds up to some $k \in \mathbb{N}^*$. Hence, we have by the Markov property
\begin{eqnarray*}
    &&\mathbb{P}_z\left( \widetilde{S}> (k+1)d_z \; | \;   Y_0 = x  \right)\\ 
    &=&\mathbb{P}_z\left(  \widetilde{S}> kd_z \;  |  \;  Y_0 = x   \right)\sum_{x'\in B_z} \mathbb{P}_z \left( \widetilde{S}> (k+1)d_z \;  |  \; \widetilde{S}> kd_z, Y_{kd_z}=x',  Y_0 = x\right) \mathbb{P}_z\left( Y_{kd}=x'\;  |  \;\widetilde{S}> kd_z,  Y_0 = x   \right)\\
    &=&\mathbb{P}_z\left(  \widetilde{S}> kd_z \;  |  \;  Y_0 = x   \right)\sum_{x'\in B_z} \mathbb{P}_z \left( \widetilde{S}>d_z \;  |  \; Y_0 = x'\right) \mathbb{P}_z\left( Y_{kd}=x'\;  |  \;\widetilde{S}> kd_z,  Y_0 = x   \right)\\
    &\leq& (1 - \varepsilon_z)^{k+1}\sum_{x'\in B_z}\left( Y_{kd_z}=x' \;  |  \; \widetilde{S}> kd_z, Y_0 = x   \right)\\
    &\leq& (1 - \varepsilon_z)^{k+1}.
\end{eqnarray*}
Let $j\in\mathbb{N}$ and consider the Euclidean division of $j$ by $d_z$, i.e. $j=kd_z+r$ with $(k,r) \in \mathbb{N} \times \{ 0, \ldots , d_z-1 \}$. Then, it turns out that 
\begin{eqnarray*}
    c_j (x,z)=\mathbb{P}_z\left(S=j,S'>j\; | \;   Y_0 = x \right) \leq \mathbb{P}_z \left(\widetilde{S}> j-1 \; | \;   Y_0 = x \right) \leq (1 - \varepsilon_z)^{\frac{j -1}{d_z} -1} = m_z r_z^j,
\end{eqnarray*}
with $m_z := (1 - \varepsilon_z)^{-\frac{d_z+1}{d_z}}$ and $r_z := (1 - \varepsilon_z)^{\frac{1}{d_z}}$.

\subsection{Proof of Proposition \ref{prop::density}}    

The discrepancy between the density function and its empirical counterpart allows us to have, for any $t\geq 0$, $x\in\mathcal{X}$ and $z\in\mathcal{Z}$
    \begin{eqnarray*}
            &&f_{T,n}^{(k)}(t,x,z)-f_T(t,x,z)\\
            &=&\sum_{j=1}^kc_{n,j}^{(k)}(x,z)\dfrac{\lambda_{n,z}^jt^{j-1}e^{-\lambda_{n,z}t}}{(j-1)!}-c_j(x,z)\dfrac{\lambda_z^jt^{j-1}e^{-\lambda_zt}}{(j-1)!}-\sum_{j\geq k+1}c_j(x,z)\dfrac{\lambda_z^jt^{j-1}e^{-\lambda_zt}}{(j-1)!}\\
            &=&\sum_{j=1}^kc_{n,j}^{(k)}(x,z) \dfrac{t^{j-1}}{(j-1)!}\left(\lambda_{n,z}^je^{-\lambda_{n,z}t}-\lambda_z^je^{-\lambda_zt}\right)+\sum_{j=1}^k\dfrac{\lambda_{z}^jt^{j-1}e^{-\lambda_{z}t}}{(j-1)!}\left(c_{n,j}^{(k)}(x,z)-c_j(x,z)\right)\\
            &&-\sum_{j\geq k+1}c_j(x,z)\dfrac{\lambda_z^jt^{j-1}e^{-\lambda_zt}}{(j-1)!}.
    \end{eqnarray*}
    This implies by convexity of the square function
    \begin{eqnarray*}
        &&\dfrac{1}{3}\left(f_{T,n}^{(k)}(t,x,z)-f_T(t,x,z)\right)^2\\
        &\leq&  \left(\sum_{j=1}^kc_{n,j}^{(k)}(x,z)\dfrac{t^{j-1}}{(j-1)!}\left(\lambda_{n,z}^je^{-\lambda_{n,z}t}-\lambda_z^je^{-\lambda_zt}\right)\right)^2+  \left(\sum_{j= 1}^k\dfrac{\lambda_{z}^jt^{j-1}e^{-\lambda_{z}t}}{(j-1)!}\left(c_{n,j}^{(k)}(x,z)-c_j(x,z)\right)\right)^2\\
        &&+ \left(\sum_{j\geq k+1}c_j(x,z)\dfrac{\lambda_z^jt^{j-1}e^{-\lambda_zt}}{(j-1)!}\right)^2\\
        &=:& r_{n,1}^{(k)}(t,x,z)+r_{n,2}^{(k)}(t,x,z)+r_{n,3}^{(k)}(t,x,z).
    \end{eqnarray*}
By applying the Cauchy-Schwarz inequality, we obtain 
\begin{eqnarray*}
    r_{n,1}^{(k)}(t,x,z)&\leq& \left(\sum_{j=1}^kc_{n,j}^{(k)}(x,z)^2 \right) \left( \sum_{j=1}^k \dfrac{t^{2j-2}}{[(j-1)!]^2}\left(\lambda_{n,z}^je^{-\lambda_{n,z}t}-\lambda_z^je^{-\lambda_zt}\right)^2 \right) \\
    &\leq& \left(\sum_{j=1}^kc_{n,j}^{(k)}(x,z) \right) \left( \sum_{j=1}^k \dfrac{t^{2j-2}}{[(j-1)!]^2}\left(\lambda_{n,z}^je^{-\lambda_{n,z}t}-\lambda_z^je^{-\lambda_zt}\right)^2 \right) \\
    &\leq&  \sum_{j=1}^k \dfrac{t^{2j-2}}{[(j-1)!]^2}\left(\lambda_{n,z}^je^{-\lambda_{n,z}t}-\lambda_z^je^{-\lambda_zt}\right)^2.
\end{eqnarray*}
where the last inequality results from the condition settled on the estimator $(c_{n,j}(x,z))_{1\leq j\leq k}$. Note $q_{n,z}=\dfrac{\lambda_{n,z}}{\lambda_z+\lambda_{n,z}}$ for convenience, successive integrations of the function $t \mapsto t^\alpha e^{-t}$ for $\alpha\in\mathbb{N}$ gives
\begin{eqnarray*}
    \int_0^{+\infty} \sup_{x \in \mathcal{X}} r_{n,1}^{(k)}(t,x,z)dt&\leq&\sum_{ j =1}^k\dfrac{(2j-2)!}{[(j-1)!]^2}\left(\dfrac{\lambda_{n,z}}{2^{2j-1}}-2 \dfrac{\lambda_z^j\lambda_{n,z}^j}{(\lambda_z+\lambda_{n,z})^{2j-1}}+\dfrac{\lambda_z}{2^{2j-1}}\right)\\
    &=&2(\lambda_z+\lambda_{n,z})\sum_{ j =1}^k \binom{2(j-1)}{j-1} \left(\dfrac{1}{2^{2j}} - (q_{n,z}(1-q_{n,z}))^j \right) \\
    &=&2(\lambda_z+\lambda_{n,z})\sum_{ j =0}^{k-1} \binom{2j}{j} \left(\dfrac{1}{4^{j+1}} - (q_{n,z}(1-q_{n,z}))^{j+1} \right) \\
    &=&2(\lambda_z+\lambda_{n,z}) \left(\dfrac{1}{4} - q_{n,z}(1-q_{n,z}) \right) \sum_{ j =0}^{k-1} \binom{2j}{j} \sum_{i=0}^j \dfrac{1}{4^{i}} (q_{n,z}(1-q_{n,z}))^{j-i} \\ 
    &\leq& 2(\lambda_z+\lambda_{n,z}) \left(\dfrac{1}{4} - q_{n,z}(1-q_{n,z}) \right) \sum_{ j =0}^{k-1} \binom{2j}{j} (j+1)\dfrac{1}{4^{j}},
\end{eqnarray*}
where the last inequality makes use of the supremum of the function $p\in [0,1]\mapsto p(1-p)$ equal to 1/4. It follows that one can verify by induction that for any $j \in \mathbb{N}$, $ \displaystyle \binom{2j}{j} \dfrac{1}{4^{j}} \leq \dfrac{1}{\sqrt{j+1}} $. This shows that 
\begin{eqnarray*}
    \sum_{ j =0}^{k-1} \binom{2j}{j} (j+1)\dfrac{1}{4^{j}} \leq (k+1)^{3/2}
\end{eqnarray*}
and yields
\begin{eqnarray}
\label{proof::rn1}
    \int_0^{+\infty} \sup_{x \in \mathcal{X}} r_{n,1}^{(k)}(t,x,z)dt 
    &\leq& \left(\dfrac{1}{2} - q_{n,z}\right)^2 2(\lambda_z+\lambda_{n,z})(k+1)^{3/2}=\dfrac{\left(\lambda_{n,z}-\lambda_z\right)^2}{2(\lambda_z+\lambda_{n,z})}(k+1)^{3/2}.
\end{eqnarray}
Likewise, we consider the second remaining term with
\begin{eqnarray*}
    r_{n,2}^{(k)}(t,x,z)&\leq& \sup_{j=1,\ldots,k}\left(c_{n,j}^{(k)}(x,z)-c_j(x,z)\right)^2\left(\sum_{j= 1}^k\dfrac{\lambda_{n,z}^jt^{j-1}e^{-\lambda_{n,z}t}}{(j-1)!}\right)^2\\
\end{eqnarray*}
so that
\begin{eqnarray*}
    \nonumber\int_0^{+\infty}\sup_{x\in\mathcal{X}}r_{n,2}^{(k)}(t,x,z)dt&\leq& \sup_{x\in\mathcal{X},j=1,\ldots,k}\left(c_{n,j}^{(k)}(x,z)-c_j(x,z)\right)^2\sum_{1\leq i,j\leq k}\dfrac{\lambda_z}{2^{i+j-1}}\dfrac{(i+j-2)!}{(i-1)!(j-1)!},
\end{eqnarray*}
where
\begin{eqnarray*}    
    \sum_{1\leq i,j\leq k}\dfrac{\lambda_z}{2^{i+j-1}}\dfrac{(i+j-2)!}{(i-1)!(j-1)!}&=& \lambda_z\dfrac{1}{2} \sum_{0\leq i,j\leq k-1}\dfrac{(i+j)!}{i!j!} \frac{1}{2^{i+j}} =   \dfrac{\lambda_z}{2} \sum_{l=0}^{2(k-1)}\sum_{0\leq i,j\leq k-1,\,i+j=l} \binom{l}{i} \frac{1}{2^{l}}  \\
    &\leq& \dfrac{\lambda_z}{2} \sum_{l=0}^{2(k-1)}\sum_{i=0}^l \binom{l}{i} \frac{1}{2^{l}} =  \lambda_z\dfrac{2k-1}{2}.
\end{eqnarray*}
This leads to the following inequality
\begin{eqnarray}
\label{proof::rn2}
   \int_0^{+\infty}\sup_{x\in\mathcal{X}}r_{n,2}^{(k)}(t,x,z)dt \leq \sup_{x\in\mathcal{X},j=1,\ldots,k}\left(c_{n,j}^{(k)}(x,z)-c_j(x,z)\right)^2\lambda_z\dfrac{2k-1}{2}.
\end{eqnarray}
We next show that the last remaining term is negligible with respect to the other rates. Indeed, Lemma \ref{lem: expo_decrease} ensures that for $j$ large enough
\begin{eqnarray*}
    \sup_{x\in\mathcal{X}}c_j(x,z)\leq m_zr_z^j
\end{eqnarray*}
so that 
\begin{eqnarray*}
   \int_0^{+\infty} \sup_{x\in\mathcal{X}}r_{n,3}^{(k)}(t,x,z)dt &\leq& \lambda_zm_z^2\sum_{ i,j\geq k+1}(r_z/2)^{i+j-1}\dfrac{(i+j-2)!}{(i-1)!(j-1)!}\\
   &=& \lambda_zm_z^2\sum_{ i,j\geq k}(r_z/2)^{i+j+1}\dfrac{(i+j)!}{i!j!}\\
    &=&\lambda_zm_z^2\sum_{ l\geq 2k}(r_z/2)^{l+1}\sum_{i,j\geq k,i+j=l}\dfrac{l!}{i!(l-i)!}\\
    &\leq&\lambda_zm_z^2\sum_{ l\geq 2k}(r_z/2)^{l+1}\sum_{i=0}^l\dfrac{l!}{i!(l-i)!}\\
    &=&\dfrac{\lambda_zm_z^2}{2}\sum_{ l\geq 2k}r_z^{l+1}=\dfrac{\lambda_zm_z^2r_z^{2k+1}}{2(1-r_z)}
\end{eqnarray*}
which concludes the first part of the result together with (\ref{proof::rn1}) and (\ref{proof::rn2}). To prove the last part of the result, note that since $A_n \subset A$, for any $1\leq j\leq k$, $x\in\mathcal{X}$ and $z\in\mathcal{Z}$
    \begin{eqnarray*}
        &&c_{n,j}^{(k)}(x,z)-c_j(x,z)\\
        &=&\sum_{x_1\in\mathcal{X}}p_{n,z}(x,x_1)c_{n,{j-1}}^{(k-1)}(x_1,z)\bold{1}_{\{x\notin A_n\}}-p_z(x,x_1)c_{j-1}(x_1,z)\bold{1}_{\{x\notin A\}}\\
        &=&\bold{1}_{\{x\notin A\}}\sum_{x_1\in\mathcal{X}}p_{n,z}(x,x_1)c_{n,{j-1}}^{(k-1)}(x_1,z)-p_z(x,x_1)c_{j-1}(x_1,z)+\bold{1}_{\{x\in A \backslash A_n\}}\sum_{x_1\in\mathcal{X}}p_{n,z}(x,x_1)c_{n,{j-1}}^{(k-1)}(x_1,z).
    \end{eqnarray*}
where
    \begin{eqnarray*}
       &&\sum_{x_1\in\mathcal{X}}p_{n,z}(x,x_1)c_{n,{j-1}}^{(k-1)}(x_1,z)-p_z(x,x_1)c_{j-1}(x_1,z)\\
        &=&\sum_{x_1\in\mathcal{X}}\left(p_{n,z}(x,x_1)-p_z(x,x_1)\right)c_{j-1}(x_1,z)+\sum_{x_1\in\mathcal{X}}p_{n,z}(x,x_1)\left(c_{n,{j-1}}^{(k-1)}(x_1,z)-c_{j-1}(x_1,z)\right).
    \end{eqnarray*}
This yields
    \begin{eqnarray*}
        |c_{n,j}^{(k)}(x,z)-c_j(x,z)|&\leq&\sup_{x,x'\in\mathcal{X}}\left|p_{n,z}(x,x')-p_z(x,x')\right|\sum_{x\in\mathcal{X}}c_{j-1}(x,z)+\bold{1}_{\{x\in A \backslash A_n\}}\sum_{x_1\in\mathcal{X}}p_{n,z}(x,x_1)c_{n,{j-1}}^{(k-1)}(x_1,z)\\
        &&+\sum_{x_1\in\mathcal{X}}p_{n,z}(x,x_1)\left|c_{n,{j-1}}^{(k-1)}(x_1,z)-c_{j-1}(x_1,z)\right|.
    \end{eqnarray*}
    In particular, multiple inductions over the index $j$ imply that
    \begin{eqnarray*}
        |c_{n,j}^{(k)}(x,z)-c_j(x,z)|&\leq&\sup_{x,x'\in\mathcal{X}}\left|p_{n,z}(x,x')-p_z(x,x')\right|\sum_{l=1}^{j}\sum_{x\in\mathcal{X}}c_{j-l}(x,z)+j\sup_{y\in\mathcal{X}}\bold{1}_{\{y\in A/A_n\}}\\
        &+&\sum_{x_1,x_2,\ldots,x_j\in\mathcal{X}}p_{n,z}(x,x_1)p_{n,z}(x_1,x_2)\ldots p_{n,z}(x_{j-1},x_j)\left|c_{n,0}^{(k-j)}(x_j;z)-c_0(x_j;z)\right|.
    \end{eqnarray*}
    However, for any integers $n,k$ and $y\in\mathcal{X}$, $c_0(y;z)-c_{n,0}^{(k)}(y;z)= \mathds{1}_{\{y\in A \backslash A_n\}}$ by construction. In particular, $\sup_{y\in\mathcal{X}}\mathds{1}_{\{y\in A \backslash A_n\}}=\mathds{1}_{\{A \backslash A_n\neq\emptyset\}}$ which implies almost surely that
    \begin{eqnarray*}
        \sup_{x\in\mathcal{X},\,j=1,\ldots,k}|c_{n,j}^{(k)}(x,z)-c_j(x,z)|^2\leq 2\left(\sup_{x,x'\in\mathcal{X}}|p_{n,z}(x,x')-p_z(x,x')|\sum_{x\in\mathcal{X}}\sum_{j=0}^{k-1}c_{j}(x,z)\right)^2+2(k+1)^2\mathds{1}_{\{A \backslash A_n\neq\emptyset\}}.
    \end{eqnarray*}
There are two possible options for the set $A \backslash A_n$ to be no empty: either $A_n$ is empty or one can find $x\in A$ such that there is no terminal state $Y_{M_j}^{(j)}=x$ with $\delta_j=1$. Formally, this implies that
\begin{eqnarray*}
    \mathbb{P}(A \backslash A_n\neq\emptyset)&=&\mathbb{P}\left(\bigcup_{x\in A}\bigcap_{j=1}^n\{Y_{S_j}^{(j)}\neq x,\delta_j=1\}\right)+\mathbb{P}(A_n=\emptyset)\\
    &\leq&\sum_{x\in A}\mathbb{P}(Y_S\neq x,\delta=1)^n+\mathbb{P}(\delta=0)^n
\end{eqnarray*}
which concludes the proof.

\subsection{Proof of Lemma \ref{lem::isolated}}
The result follows from the definition of the temporal difference of the estimator $c_{n,j}^{(k)}$. Indeed, for any $x\in\mathcal{X}$ and $z\in\mathcal{Z}$, we have by direct induction when $k\geq j$
\begin{eqnarray*}
    c_{n,j}^{(k)}(x,z)&=&\sum_{x_1,x_2,\ldots,x_{j-1}\in\mathcal{X} \backslash A_n,\,x_j\in A_n}p_{n,z}(x,x_1)\ldots p_{n,z}(x_{j-1},x_{j})c_{n,0}^{(k-j)}(x_j,z)\\
    &=&\sum_{x_1,x_2,\ldots,x_{j-1}\in\mathcal{X} \backslash A_n,\,x_j\in A_n}p_{n,z}(x,x_1)\ldots p_{n,z}(x_{j-1},x_{j})\bold{1}_{x_j\in A_n}.
\end{eqnarray*}
If $x\in C_z$, then there exists no path from $x$ to $A_n$ since $A_n\subset A$. For any combination of states $(x,x_1,\ldots,x_j)$ with $x_j\in A_n$, this implies that there exist $1\leq\ell\leq j$ such that $p_z(x_\ell,x_{\ell+1})=p_{n,z}(x_\ell,x_{\ell+1})=0$ and proves the result. 

\subsection{Proof of Lemma \ref{lem::curerate}}
Similarly than the previous results, the proof depends largely on the initial position of the chain. For any $x\in A_n$, we have $x\in A$ so that the estimator construction directly returns  $R_n^{(k)}(x,z)=\mathbb{P}_z(T=+\infty|Y_0=x)=0$ for any $n\in\mathbb{N}^*$.
If $x\notin A_n$ but $x\in A$, then 
\begin{eqnarray*}
    R_n^{(k)}(x,z)-\mathbb{P}_z(T=+\infty|Y_0=x)=R_n^{(k)}(x,z)\leq \bold{1}_{\{A \backslash A_n\neq\emptyset\}}.
\end{eqnarray*}
Hence the estimator's risk is bounded above by $\mathbb{P}(x\in A \backslash A_n)$ which returns the same rate than that of the end of the proof of Proposition \ref{prop::density}. Next, if neither $x$ belongs to $A_n$ nor $A$, we have by the estimator's construction that
\begin{eqnarray*}
    R_n^{(k)}(x,z)-\mathbb{P}_z(T=+\infty|Y_0=x)&=&\sum_{0\leq j\leq k}c_j(x)-c_{n,j}^{(k)}(x)+\sum_{j>k}c_j(x)
\end{eqnarray*}
hence yielding that
\begin{eqnarray*}
    &&\mathbb{E}\left[\left(\sup_{x\in\mathcal{X}}R_n^{(k)}(x,z)-\mathbb{P}_z(T=+\infty|Y_0=x)\right)^2\right]\\
    &\leq& 2k^2 \mathbb{E}\left[\sup_{x\in\mathcal{X},\,j=1,\ldots,k}\left(c^{(k)}_{n,j}(x,z)-c_j(x,z)\right)^2\right]+2\mathbb{P}(k<S<+\infty)^2.
\end{eqnarray*}
By Lemma \ref{lem: expo_decrease}, we finally have that for $k$ large enough
\begin{eqnarray*}
    \mathbb{P}(k<S<+\infty)\leq m_z\sum_{j>k}r_z^j=m_z\dfrac{r_z^{k+1}}{1-r_z}.
\end{eqnarray*}

\subsection{Technical lemmas}

Before presenting the proofs of the main results, we introduce some technical lemmas that will be useful later on. We will consider deviations given in Lemma \ref{lem:deviationRhx} for the following process 
\begin{eqnarray*}\label{eq:processRH}
    R_{h,U}^z = \frac{1}{n} \sum_{i=1}^n \frac{K_h(z-Z_i)U_i}{\E \left[ K_h(z-Z)U\right]}
\end{eqnarray*}
where $h\in(0,1]$ and $(U_i)_{1\leq i\leq n}$ is a collection of positive \textit{i.i.d.} random variables drawn from the same distribution than that of $U$. Lemma \ref{lem:smallballHolderversion} provides concentration results for small ball probabilities with covariates satisfying the Hölder assumption $(\mathcal{H})$ and similar results with kernel expectations. Finally, Lemma \ref{lem:deviationRhx} gives upper bound probabilities for the process $R_{h,U}^z$ to be away from one. 

\begin{lem}\label{lem:smallballHolderversion}
Assume $(\mathcal{H})$. Let $h\in (0,1)$ and $z\in\text{int}(\mathcal{Z})$. Then, there exists a measurable function $z\mapsto\psi^z_n$ such that
\begin{eqnarray*}
        \P( \|z-Z\|\le h)=h^p(vf_Z(z)+\psi^z_n)
\end{eqnarray*}
where $\sup_z|\psi^z_n|=\mathcal{O}(h^\alpha)$ and $v=\int_{B(0,1)}du$ and $B(0,1)=\{u\in\mathbb{R}^p,\,\Vert u\Vert\leq 1\}$. For any $l\in\mathbb{N}$ such that there exist positive constants $m_l<M_l$ with $m_l\leq\mathbb{E}\left[U^l|Z\right]\leq M_l$ a.s., we have
    \begin{eqnarray*}
    m_l[c_Kh^{-p}]^l  \le&\dfrac{\E \left[\left(K_h\left(z-Z \right)U\right)^l \right]}{\P \left( \| z-Z \| \le h \right)}& \le M_l[C_K h^{-p}]^l 
\end{eqnarray*}
\end{lem}

\begin{proof}
    Proof of the concentration for the small ball probabilities can be found in \cite{Egea25} in Lemma A.1. In order to prove the inequalities concerning the kernel expectations, it is sufficient to see that
    \begin{eqnarray*}
        \E \left[K_h\left(z-Z \right)^lU^l \right]  = h^{-pl} \E\left[ K\left(\frac{z-Z}{h} \right)^l \mathbb{E}\left[U^l|Z\right] \right].
    \end{eqnarray*}
    The function $K$ is assumed compactly supported on the unit ball, which gives us
        \begin{equation*}
            \E \left[K_h\left(z-Z \right)^l U^l\right]  = h^{-pl} \E\left[ \ind_{\{\|z-Z\|\le h\}}K\left(\frac{z-Z}{h} \right)^l \mathbb{E}\left[U^l|Z\right]\right],
        \end{equation*}
    and the first assertion follows since $K$ is bounded.
\end{proof}

\begin{lem}
\label{lem:kernellimit}
    Let $g:\mathcal{Z}\to\mathbb{R}$ be any bounded and Hölder continuous function of order $\alpha'\in(0,1]$ and constant $C_g>0$, then for any $z\in\mathcal{Z}$
    \begin{eqnarray*}
        \mathbb{E}[K_h(z-Z)g(Z)]=\int_\mathcal{Z}f_Z(u)K_h(z-u)g(u)du=f_Z(z)g(z)+O(h^{\alpha\wedge\alpha'}).
    \end{eqnarray*}
In particular, for $n$ large enough
\begin{eqnarray*}
    \dfrac{\left|\mathbb{E}[K_h(z-Z)g(Z)]-f_Z(z)g(z)\right|}{h^{\alpha\wedge\alpha'}}\leq C g(z) + C_gf_Z(z) + CC_g.
\end{eqnarray*}
\end{lem}
\begin{proof}
The proof mostly relies on integration substitute change and the Hölder condition of the functions $f_Z$ and $g$. Indeed, for $n$ large enough
\begin{eqnarray*}
    \int_\mathcal{Z}f_Z(u)K_h(z-u)g(u)du&=&\int_{B(0,1)}K(u)f_Z(z-hu)g(z-hu)du\\
    &=&\int_{B(0,1)}K(u)\left(f_Z(z-hu)g(z-hu)-f_Z(z)g(z)\right)du + f_Z(z)g(z)
\end{eqnarray*}
where
\begin{eqnarray*}
    &&\left|\int_{B(0,1)}K(u)\left(f_Z(z-hu)g(z-hu)-f_Z(z)g(z)\right)du\right|\\
    &\leq& \int_{B(0,1)}K(u)\left|f_Z(z-hu)g(z-hu)-f_Z(z)g(z)\right|du\\
    &\leq&\int_{B(0,1)}K(u)\left(g(z-hu)|f_Z(z-hu)-f_Z(z)|+f_Z(z) |g(z)-g(z-hu)|\right)du\\
    &\leq& \int_{B(0,1)}K(u)\Vert u\Vert^{\alpha\wedge\alpha'}du \left(C g(z) h^\alpha+C_gf_Z(z) h^{\alpha'}+CC_gh^{\alpha+\alpha'}\right)\\
    &\leq&C g(z) h^\alpha+C_gf_Z(z) h^{\alpha'}+CC_gh^{\alpha+\alpha'}.
\end{eqnarray*}
\end{proof}

\bigskip

\begin{lem}\label{lem:deviationRhx}Let $h\in (0,1)$ $z\in\text{int}(\mathcal{Z})$. Assume $(\mathcal{H})$ and that for any $l\in\mathbb{N}$, there exist some positive constants $M_l$, $d$ and $D$ such that $\mathbb{E}\left[U^l|Z\right]\leq M_l$ a.s. with $M_l\leq Dl!d^l$. Also, assume that one can find a constant $M_0>0$ such that $M_0\leq\mathbb{E}\left[U|Z\right]$ almost surely. Then the following inequality holds
    \begin{equation*}
         \P\left(R^z_{h,U} \le 1/2 \right)\vee\P\left(R^z_{h,U} \ge 3/2 \right) \le  2 \exp\left(-c_Unh^p(vf_Z(z)+\psi^z_n)\right)
    \end{equation*}
where $\psi_n^z$ is defined in Lemma \ref{lem:smallballHolderversion} and $c_U:=\left(8\left(\frac{Dd^2C_K^2}{M_0^2c_K^2}+\frac{dC_K}{2M_0c_K}\right)\right)^{-1}$.
\end{lem}
\begin{proof}
     We prove the above inequality based on the event inclusion of the subsets $\{R^z_{h,U} \le 1/2\}$ and $\{R^z_{h,U} \ge 3/2\}$ into $\{|R^z_{h,U} - 1 | \ge 1/2\}$. The deviation will be controlled thanks to Bernstein's inequality (see \cite[Lemma 8]{Mas98}) which we state as followed. Let $(V_k)_{k=1}^n$ be independent random variables and $S_n(V) = \sum_{i=1}^n (V_i - \E [V_i])$. Then for any $b$ and $w$ positive constants with
     \begin{equation*}
         \forall l \ge 2, \quad \frac{1}{n}\sum_{i=1}^n \E [|V_i|^l] \le \frac{l!}{2}w^2 b^{l-2} , 
     \end{equation*}
     we have for $\eta>0$
     \begin{eqnarray}
     \label{eq:bernstein}
         \P \left(\frac{1}{n}|S_n(V)| \ge \eta \right) \le 2 \exp\left(-\frac{n\eta^2}{2(w^2+b\eta)}\right).
     \end{eqnarray}
     We use this result in our context by considering the random variables 
     \begin{eqnarray*}
        V_i:=\frac{K_h(z-Z_i)U_i}{\E[K_h(z-Z)U]},\quad i=1,\ldots,n 
     \end{eqnarray*}
     with $S_n(V)/n=R_{h,U}^x -1$. By Lemma \ref{lem:smallballHolderversion}, we have for $l\ge2$
    \begin{equation*}
         \E \left[|V_i|^l \right]= \frac{ \E \left[K_h^l (z - Z)U^l\right]}{\E \left[K_h(z-Z)U\right]^l}\le \frac{M_lC_K^l}{M_0^lc_K^l \P( \|z-Z\|\le h)^{l-1}},
     \end{equation*}
     which implies that
    \begin{equation*}
         \frac{1}{n}\sum_{i=1}^n \E [|V_i|^l] \le  \dfrac{l!}{2}\underbrace{\frac{2Dd^2C_K^2}{M_0^2c_K^2 \P( \|x-X\|\le h)}}_{\eqqcolon w^2}  \underbrace{\left( \frac{dC_K}{M_0c_K \P( \|x-X\|\le h)}\right)^{l-2}}_{\eqqcolon b^{l-2}},
     \end{equation*}
     The result finally follows from (\ref{eq:bernstein}) with $\eta=1/2$ and applying Lemma \ref{lem:smallballHolderversion}.
\end{proof}
\bigskip

\subsection{Proof of Theorem \ref{theorem::kernelmethod}}

The main idea of the proof consists on derivation of the rates of convergence for the estimators $\lambda_{n,z}$ and $(c_{n,j}^{(k)}(x))_{j\geq 0}$ in order to apply Proposition \ref{prop::density}.

\paragraph{Transition rate estimator.} 
We denote 
\begin{eqnarray*}
    U_j=\dfrac{E^{(j)}_{M_j}}{M_j}=\dfrac{\sum_{i=1}^{M_j}E_i^{(j)}-E_{i-1}^{(j)}}{M_j},\quad j=1,\ldots,n.
\end{eqnarray*}
As mentioned in the estimator presentation, we limit large outcomes when they exceed a given threshold. In particular, this allows to restrict the analysis on the set $\{R_{h,U}^z> 1/2\}$ since
\begin{eqnarray}
\label{proof::transitionratemaj}
    \nonumber&&\hspace{-1cm}\mathbb{E}\left[(\lambda_{n,z}-\lambda_z)^2\right]\\
    \nonumber&\leq&\mathbb{E}\left[(\lambda_{n,z}-\lambda_z)^2\mathds{1}_{\{R_{h,U}^z> 1/2,R_{h,1}^z< 3/2\}}\right]+(\tilde{\lambda}-\lambda_{\min})^2\mathbb{P}(\{R_{h,U}^z\leq 1/2\}\cup\{R_{h,1}^z\geq 3/2\})\\
    \nonumber &\leq&\mathbb{E}\left[\left(\left(\sum_{j=1}^n\omega_{h,z}(Z_j)U_j\right)^{-1}-\lambda_z\right)^2\mathds{1}_{\{R_{h,U}^z> 1/2,R_{h,1}^z< 3/2\}}\right]\\
    &&+(\tilde{\lambda}-\lambda_{\min})^2(\mathbb{P}(R_{h,U}^z\leq 1/2)+\mathbb{P}(R_{h,1}^z\geq 3/2)).
\end{eqnarray}
We can use Lemma \ref{lem:deviationRhx} to show that the probability that $R_{h,U}^z \leq 1/2$ rapidly decreases when $n$ is large enough. To do so, we check that the distribution of the $U_j$'s fits the conditions of the latter result. Indeed, we have for any $l\in\mathbb{N}$ and $z\in\mathcal{Z}$
\begin{eqnarray*}
    \mathbb{E}_z\left[U^l\right]=\mathbb{E}_z\left[\dfrac{\left(\sum_{i=1}^{M}E_i-E_{i-1}\right)^l}{M^l}\right]
\end{eqnarray*}
where given $M=m\in\mathbb{N}^*$ and $Z=z$, the random sum $\sum_{i=1}^{M}E_i-E_{i-1}$ follows an Erlang distribution with parameters $M$ and $\lambda_z$. This implies that
\begin{eqnarray*}
    \mathbb{E}_z\left[\left.\left(\sum_{i=1}^{M}E_i-E_{i-1}\right)^l\right|M=m\right]&=&\dfrac{1}{(m-1)!}\int_0^{+\infty}\lambda_z^{m}u^{l+m-1}e^{-\lambda_zu}du =\dfrac{1}{\lambda_z^l(m-1)!}\int_0^{+\infty}u^{l+m-1}e^{-u}du\\
    &=&\dfrac{(l+m-1)!}{\lambda_z^l(m-1)!} =l!\binom{l+m-1}{m-1}\lambda_z^{-l}.
\end{eqnarray*}
Note that for any $m\in\mathbb{N}^*$, we have
\begin{eqnarray*}
    \mathbb{P}_z(M=m)&=&\mathbb{P}_z(S=m,S\leq L)+\mathbb{P}_z(L=m,L<S)\\
    &\leq&\mathbb{P}_z(m\leq L)+\mathbb{P}_z(L=m)\\
    &\leq&2\mathbb{P}_z(m\leq L).
\end{eqnarray*}
Consequently, one can easily verify that the conditions of Lemma \ref{lem:deviationRhx} are verified with $M_0=\lambda_{\max}^{-1}$, $M_1 = \lambda_{\min}^{-1}$ and for $l\geq 2$
\begin{eqnarray*}
\label{proof:Mcondition}
    E_z[U^l]&\leq& 2l!\sum_{m\geq1}\binom{l+m-1}{m-1}\dfrac{\mathbb{P}_z(m\leq L)}{(m\lambda_z)^{l}}\\
    &\leq&2l!\sup_{z\in\mathcal{Z}}\mathbb{E}_z[L]\lambda_{\min}^{-l}:=M_l.
\end{eqnarray*}
Focusing on the main estimator's part, we obtain the following decomposition
\begin{eqnarray*}
  \left(\sum_{j=1}^n\omega_{h,z}(Z_j)U_j\right)^{-1}  &=&\dfrac{1}{n}\sum_{j=1}^nK_h(z-Z_j)\left(\dfrac{1}{n}\sum_{j=1}^nK_h(z-Z_j)U_j\right)^{-1}\\
    &=&R_{h,1}^z\left(\dfrac{\mathbb{E}[K_h(z-Z)]}{\mathbb{E}[K_h(z-Z)U]}-\mathbb{E}[K_h(z-Z)]\dfrac{\dfrac{1}{n}\sum_{j=1}^nK_h(z-Z_j)U_j-\mathbb{E}[K_h(z-Z)U]}{R_{h,U}^z\mathbb{E}[K_h(z-Z)U]^2}\right)
\end{eqnarray*}
which implies that
\begin{eqnarray*}
    &&\left(\sum_{j=1}^n\omega_{h,z}(Z_j)U_j\right)^{-1}  -\lambda_z\\
     &=& \left( \dfrac{\mathbb{E}[K_h(z-Z)]}{\mathbb{E}[K_h(z-Z)U]}-\lambda_z \right) + \left( \dfrac{\mathbb{E}[K_h(z-Z)]}{\mathbb{E}[K_h(z-Z)U]}(R_{h,1}^z-1) \right)\\
     &&-R_{h,1}^z\mathbb{E}[K_h(z-Z)]\dfrac{\dfrac{1}{n}\sum_{j=1}^nK_h(z-Z_j)U_j-\mathbb{E}[K_h(z-Z)U]}{R_{h,U}^z\mathbb{E}[K_h(z-Z)U]^2}\\
     &:=&R_{n,1}+R_{n,2}+R_{n,3}.
\end{eqnarray*}
Due to the convexity of the square function, we can see that we can obtain an upper bound of the quadratic error based on the second-order moment of each of the remaining terms independently. We begin by examining the conditional expectation of the transition rate estimator. This is the primary source of bias in the discrepancy between the estimator and the true value. We thus have 
\begin{eqnarray*}
    \mathbb{E}[K_h(z-Z)U]=\int_\mathcal{Z}f_Z(u)K_h(z-u)\mathbb{E}_u[U]du,
\end{eqnarray*}
where for any $u\in\mathcal{Z}$
\begin{eqnarray*}
    \mathbb{E}_u[U]=\mathbb{E}_u\left[\dfrac{1}{M}\mathbb{E}_u\left[\left.\sum_{i=1}^ME_i-E_{i-1}\right|M\right]\right]=\lambda_u^{-1} ,
\end{eqnarray*}
since the random series $(E_i-E_{i-1})_{i\geq 1}$ forms an \textit{i.i.d.} sample of exponential random variables with parameter $\lambda_u$ independent from $S$ and $L$. By use of Lemma \ref{lem:kernellimit}, we thus have for $n$ large enough
\begin{eqnarray*}
    \left|R_{n,1}\right|&=&\left|\dfrac{\mathbb{E}[K_h(z-Z)]-f_Z(z)+f_Z(z)}{\mathbb{E}[K_h(z-Z)U]}-\lambda_z\right|\\
    &\leq&\dfrac{\left|\mathbb{E}[K_h(z-Z)]-f_Z(z)\right|+\lambda_z\left|\mathbb{E}[K_h(z-Z)U]-f_Z(z)\lambda_z^{-1}\right|}{\mathbb{E}[K_h(z-Z)U]}\\
    &\leq&C\left(1+C+f_Z(z)\left(1+\dfrac{\lambda_z}{\lambda_{\min}^2}\right)\right)\dfrac{h^\alpha}{\mathbb{E}[K_h(z-Z)U]}
    \end{eqnarray*}
since the function $z\mapsto \lambda_z^{-1}$ is Hölder continuous of order $\alpha$ and constant $\frac{C}{\lambda_{\min}^2}$. The control of the first term is then given by
\begin{eqnarray}
\label{proof::transitionrn1}
    \nonumber|R_{n,1}|&\leq& C\left(1+C+f_Z(z)\left(1+\dfrac{\lambda_z}{\lambda_{\min}^2}\right)\right)h^\alpha\left(\dfrac{1}{f_Z(z)\lambda_z^{-1}}+\dfrac{|f_Z(z)\lambda_z^{-1}-\mathbb{E}[K_h(z-Z)U]|}{f_Z(z)\lambda_z^{-1}\mathbb{E}[K_h(z-Z)U]}\right)\\
    \nonumber&\leq& \dfrac{2C}{f_Z(z)\lambda_z^{-1}}\left(1+C+f_Z(z)\left(1+\dfrac{\lambda_z}{\lambda_{\min}^2}\right)\right)h^\alpha\\
    &\leq& \dfrac{2C\lambda_{\max}}{f_Z(z)}\left(1+C+f_Z(z)\left(1+\dfrac{\lambda_{\max}}{\lambda_{\min}^2}\right)\right)h^\alpha
\end{eqnarray}
We move on to the second term and observe that $R_{n,2}=(R_{n,1}+\lambda_z)(R_{h,1}^z-1)$, which implies by (\ref{proof::transitionrn1}) and for $n$ large enough that
\begin{eqnarray*}
    (R_{n,2})^2&\leq& \left[\dfrac{2C\lambda_{\max}}{f_Z(z)}\left(1+C+f_Z(z)\left(1+\dfrac{\lambda_z}{\lambda_{\min}^2}\right)\right)h^\alpha+\lambda_z\right]^2\left(R_{h,1}^z-1\right)^2\\
    &\leq&(1+\lambda_{\max})^2\left(R_{h,1}^z-1\right)^2.
\end{eqnarray*}
The upper bound of $(R_{n,2})^2$ then relies on $\mathbb{E}\left[(R_{h,1}^z-1)^2\right]\leq\dfrac{\mathbb{E}[K_h(z-Z)^2]}{n\mathbb{E}[K_h(z-Z)]^2}$. By of use of Lemma \ref{lem:smallballHolderversion}, we obtain that
\begin{eqnarray*}
    \dfrac{\mathbb{E}[K_h(z-Z)^2]}{\mathbb{E}[K_h(z-Z)]^2}\leq \dfrac{h^{-p}}{vf_Z(z)+\psi_n^z}\dfrac{C_K^2}{c_K^2} ,
\end{eqnarray*}
which allows to have the second inequality
\begin{eqnarray}
    \label{proof::transitionrn2}
\mathbb{E}[(R_{n,2})^2]\leq (nh^p)^{-1}\dfrac{C_K^2(1+\lambda_{\max})^2}{c_K^2(vf_Z(z)+\psi_n^z)}.
\end{eqnarray}
We finally obtain the upper bound of the last term using the same analysis as the previous one, except that we must control the random processes $R_{h,1}^z$ and $R_{h,U}^z$. It appears that
\begin{eqnarray}
\label{proof::transitionrn3}
    &&\nonumber\hspace{-1cm}\mathbb{E}\left[(R_{n,3})^2\mathds{1}_{\{R_{h,U}^z> 1/2,R_{h,1}^z<3/2\}}\right]\\
    \nonumber&\leq&\dfrac{9}{4\lambda_{\min}^4(vf_Z(z)+\psi_n^z)^2}\mathbb{E}\left[\left(\dfrac{\dfrac{1}{n}\sum_{j=1}^nK_h(z-Z_j)U_j-\mathbb{E}[K_h(z-Z)U]}{R_{h,U}^z}\right)^2\mathds{1}_{\{R_{h,U}^z> 1/2\}}\right]\\
    \nonumber&\leq &\dfrac{9}{n\lambda_{\min}^4(vf_Z(z)+\psi_n^z)^2}\mathbb{E}\left[(K_h(z-Z)U)^2\right]\\
    &\leq& (nh^p)^{-1}\dfrac{36\sup_{z\in\mathcal{Z}}\mathbb{E}_z[L]}{\lambda_{\min}^6(vf_Z(z)+\psi_n^z)}.
\end{eqnarray}

\bigskip
Back to the initial upper bound, combining  (\ref{proof::transitionrn1}) with (\ref{proof::transitionrn2}) and (\ref{proof::transitionrn3}) in (\ref{proof::transitionratemaj}), we finally obtain that
\begin{eqnarray}
\label{proof:transitionfinalrate}
    \mathbb{E}\left[(\lambda_{n,z}-\lambda_z)^2\right]\leq a_1^zh^{2\alpha}+a_2^z(nh^p)^{-1}+2(\tilde{\lambda}-\lambda_{\min})^2\exp\left(-\dfrac{nh^p(vf_Z(z)+\psi_n^z)}{8\left(\frac{2\sup_{z\in\mathcal{Z}}\mathbb{E}_z[L]\lambda_{\max}^2C_K^2}{\lambda_{\min}^2c_K^2}+\frac{\lambda_{\max}C_K}{2\lambda_{\min}c_K}\right)}\right) ,
\end{eqnarray}
where 
\begin{eqnarray*}
    \left\{\begin{matrix}
        a_1^z&=&\dfrac{12C^2\lambda_{\max}^2}{f_Z(z)^2}\left(1+C+f_Z(z)\left(1+\dfrac{\lambda_{\max}}{\lambda_{\min}^2}\right)\right)^2,\\[.5cm]
        a_2^z&=&\dfrac{1}{vf_Z(z)+\psi_n^z}\left(\dfrac{C_K^2(1+\lambda_{\max})^2}{c_K^2}+\dfrac{36\sup_{z\in\mathcal{Z}}\mathbb{E}_z[L]}{\lambda_{\min}^6}\right).
    \end{matrix}
    \right.
\end{eqnarray*}

\bigskip
\paragraph{Hitting-time coefficients.}
Based on the bias-variance decomposition of $p_{n,z}(x,x')-p_z(x,x')$, we can rewrite the latter as follows
\begin{eqnarray*}
    &&(R_{h,N^x}^z)^{-1}\left(\dfrac{ \frac{1}{n}\sum_{j=1}^nK_h(z-Z_j)N_j^{x,x'}-\mathbb{E}[K_h(z-Z)N^{x,x'}]}{\mathbb{E}[K_h(z-Z)N^x]}+(1-R_{h,N^x}^z)p_z(x,x')\right.\\
    &&\left.+\dfrac{\mathbb{E}[K_h(z-Z)N^{x,x'}]}{\mathbb{E}[K_h(z-Z)N^x]}-p_z(x,x')\right)\\
    &:=&S_{n,1}(x,x')+S_{n,2}(x,x')+S_{n,3}(x,x').
\end{eqnarray*}
The same strategy used for the transition estimator applies to considering the weight terms. On the set $\{R_{h,N^x}^z> 1/2\}$, we have for any $x,x'\in\mathcal{X}$
\begin{eqnarray*}
    \mathbb{E}\left[\left(S_{n,1}(x,x')\right)^2\mathds{1}_{\{R_{h,N^x}^z> 1/2\}}\right]&\leq& 4\dfrac{\mathbb{E}\left[(K_h(z-Z)N^{x,x'})^2\right]}{n\mathbb{E}[K_h(z-Z)N^x]^2}\\
    &\leq& 4\dfrac{\mathbb{E}\left[(K_h(z-Z)N^{x})^2\right]}{n\mathbb{E}[K_h(z-Z)N^x]^2}.
\end{eqnarray*}
By definition, we almost surely have $N^x \leq L$ which ensures that $\sup_z \mathbb{E} [(N^x)^2] < + \infty$. Furthermore,  the condition with regards to $\inf_z \mathbb{E}_z[N^x]$ and Lemma \ref{lem:smallballHolderversion} implies that 
\begin{eqnarray*}
    \dfrac{\mathbb{E}\left[(K_h(z-Z)N^{x})^2\right]}{\mathbb{E}[K_h(z-Z)N^x]^2} &\leq& h^{-p}\dfrac{\sup_{z \in \mathcal{Z}} \mathbb{E}\left[L^2\right] C_K^2}{a^2c_K^2(vf_Z(z)+\psi_n^z)}.
\end{eqnarray*} 
It next follows that the supremum of the remaining term is bounded above by
\begin{eqnarray*}
    \mathbb{E}[\sup_{x,x'\in\mathcal{X}}\left(S_{n,1}(x,x')\right)^2\mathds{1}_{\{R_{h,N^x}^z> 1/2\}}]\leq \sum_{x,x\in\mathcal{X}}\mathbb{E}[\left(S_{n,1}(x,x')\right)^2\mathds{1}_{\{R_{h,N^x}^z> 1/2\}}]\leq 4| \mathcal{X} |^2  \dfrac{\mathbb{E}\left[(K_h(z-Z)N^{x})^2\right]}{n\mathbb{E}[K_h(z-Z)N^x]^2} ,
\end{eqnarray*}
which yields
\begin{eqnarray}
\label{proof::coefsn1}
    \mathbb{E}[\sup_{x,x'\in\mathcal{X}}\left(S_{n,1}(x,x')\right)^2\mathds{1}_{\{R_{h,N^x}^z> 1/2\}}]\leq 4| \mathcal{X} |^2 (nh^p)^{-1}\dfrac{ \sup_{z \in \mathcal{Z}} \mathbb{E}\left[L^2\right]C_K^2}{a^2c_K^2(vf_Z(z)+\psi_n^z)}.
\end{eqnarray}
Next, we observe that for any $x,x'\in\mathcal{X}$
\begin{eqnarray*}
    \mathbb{E}\left[\left(S_{n,2}(x,x')\right)^2\right]\leq\mathbb{E}\left[(R_{h,N^x}^z-1)^2\right]&\leq&\dfrac{\mathbb{E}[(K_h(z-Z)N^x)^2]}{n\mathbb{E}[K_h(z-Z)N^x]^2}\\
    &\leq&(nh^p)^{-1}\dfrac{ \sup_{z \in \mathcal{Z}} \mathbb{E}\left[L^2\right] C_K^2}{a^2c_K^2(vf_Z(z)+\psi_n^z)}.
\end{eqnarray*}
In a similar manner than $S_{n,1}$, this allows to control the second remaining term with
\begin{eqnarray}
\label{proof::coefsn2}
    \mathbb{E}\left[\sup_{x,x'\in\mathcal{X}}S_{n,2}^2(x,x')\right] \leq | \mathcal{X} | (nh^p)^{-1}\dfrac{ \sup_{z \in \mathcal{Z}} \mathbb{E}\left[L^2\right] C_K^2}{a^2c_K^2(vf_Z(z)+\psi_n^z)}.
\end{eqnarray}

Finally, we consider the last remaining term with
\begin{eqnarray*}
    |S_{n,3}(x,x')|&=&\left|\dfrac{\mathbb{E}[K_h(z-Z)N^xp_Z(x,x')]}{\mathbb{E}[K_h(z-Z)N^x]}-p_z(x,x')\right|\\
    &\leq&\dfrac{\mathbb{E}[K_h(z-Z)N^x|p_Z(x,x')-p_z(x,x')|]}{\mathbb{E}[K_h(z-Z)N^x]}
\end{eqnarray*}
For $n$ large enough such that $z-hu\in int(\mathcal{Z})$ for any $u\in B(0,1)$, we have
\begin{eqnarray*}
    \mathbb{E}[K_h(z-Z)N^x|p_Z(x,x')-p_z(x,x')|]&=&\int_\mathcal{Z}K_h(z-u)\mathbb{E}_u[N^x]|p_u(x,x')-p_z(x,x')|f_Z(u)du\\
    &=&\int_{B(0,1)}K(u)\mathbb{E}_{z-hu}[N^x]|p_{z-hu}(x,x')-p_z(x,x')|f_Z(z-hu)du\\
    &\leq& Ch^\alpha\int_{B(0,1)}K(u)\mathbb{E}_{z-hu}[N^x]\Vert u \Vert^\alpha f_Z(z-hu)du\\
    &\leq& Ch^\alpha\mathbb{E}[K_h(z-Z)N^x]
\end{eqnarray*}
which implies that
\begin{eqnarray}
\label{proof::coefsn3}
    \sup_{x,x'\in\mathcal{X}}\left(S_{n,3}(x,x')\right)^2\leq C^2h^{2\alpha}.
\end{eqnarray}
Because $|p_{n,z}(x,x')-p_z(x,x')|$ is uniformly smaller than 1, it only remains to show that $\mathbb{E}[\sup_{x\in\mathcal{X}}\mathds{1}_{\{R^z_{h,N^x}\leq 1/2\}}]$ is sufficiently small when $n$ grows. In our case, the number of states is finite, so that
\begin{eqnarray*}
    \mathbb{E}[\sup_{x\in\mathcal{X}}\mathds{1}_{\{R^z_{h,N^x}\leq 1/2\}}]\leq|\mathcal{X}|\sup_{x\in\mathcal{X}}\mathbb{P}\left(R^z_{h,N^x}\leq \dfrac{1}{2}\right).
\end{eqnarray*}
It turns out that we can apply Lemma \ref{lem:deviationRhx} with $U=N^x$ since $N^x \leq  L$ and $\inf_{x\in\mathcal{X},z\in\mathcal{Z}}\mathbb{E}_z[N^x]>a$. This, combined with (\ref{proof::coefsn1}), (\ref{proof::coefsn2}) and (\ref{proof::coefsn3}), proves that
\begin{eqnarray} \label{eq::14} \nonumber
    \mathbb{E}\left[\sup_{x,x'\in\mathcal{X}}(p_{n,z}(x,x')-p_z(x,x'))^2\right]&\leq& 2C^2h^{2\alpha}+ \dfrac{( 4|\mathcal{X} |^2 + |\mathcal{X} |)\sup_{z \in \mathcal{Z}} \mathbb{E}\left[L^2\right]C_K^2}{a^2c_K^2(vf_Z(z)+\psi_n^z)}(nh^p)^{-1}\\
    &&+|\mathcal{X}|\exp\left(-\dfrac{nh^p(vf_Z(z)+\psi_n^z)}{8\left(\frac{D_L C_K^2}{a^2c_K^2}+\frac{C_K}{2ac_K}\right)}\right) .
\end{eqnarray}
The proof follows from Proposition \ref{prop::density} and Inequalities (\ref{theo::temporalinequality}), (\ref{proof:transitionfinalrate}) and (\ref{eq::14}).
\end{document}